\documentclass{amsart}

\usepackage[top=3cm, bottom=3cm, left=3cm, right=3cm]{geometry}
\usepackage{subfig}
\usepackage{tikz-cd}
\usepackage{tikz}
\usepackage{caption}
\usepackage{amssymb}
\usepackage[all]{xy}
\usepackage{amsthm}
\usepackage{hyperref}

\DeclareMathOperator{\N}{\mathbf{N}}
\DeclareMathOperator{\Z}{\mathbf{Z}}
\DeclareMathOperator{\R}{\mathbf{R}}

\DeclareMathOperator{\dd}{\mathbf{d}}

\DeclareMathOperator{\Hom}{\textup{Hom}}
\DeclareMathOperator{\GL}{\textup{GL}}

\DeclareMathOperator{\Rep}{\textup{Rep}}
\DeclareMathOperator{\Ext}{\textup{Ext}}
\DeclareMathOperator{\C}{\mathbf{C}}
\DeclareMathOperator{\Q}{\mathbf{Q}}
\DeclareMathOperator{\F}{\mathbf{F}}

\DeclareMathOperator{\HH}{\mathbf{H}}

\DeclareMathOperator{\PP}{\mathbf{P}}

\DeclareMathOperator{\nil}{\textup{nil}}

\DeclareMathOperator{\U}{\mathbf{U}}
\DeclareMathOperator{\cusp}{\textup{cusp}}

\DeclareMathOperator{\CC}{\mathcal{C}}

\DeclareMathOperator{\ind}{\textup{ind}}

\DeclareMathOperator{\M}{\mathcal{M}}
\DeclareMathOperator{\Exp}{\textup{Exp}}

\DeclareMathOperator{\im}{\textup{im}}
\DeclareMathOperator{\supp}{\textup{supp}}

\DeclareMathOperator{\indec}{\textup{indec}}

\DeclareMathOperator{\End}{\textup{End}}

\DeclareMathOperator{\OO}{\mathcal{O}}
\DeclareMathOperator{\B}{\mathcal{B}}

\DeclareMathOperator{\id}{\textup{id}}

\DeclareMathOperator{\inv}{\textup{inv}}

\DeclareMathOperator{\Ob}{\textup{Ob}}
\DeclareMathOperator{\homm}{\textup{hom}}
\DeclareMathOperator{\ext}{\textup{ext}}
\DeclareMathOperator{\Fun}{\textup{Fun}}
\DeclareMathOperator{\PC}{\mathcal{P}}
\DeclareMathOperator{\IC}{\mathcal{I}}
\DeclareMathOperator{\RC}{\mathcal{R}}
\DeclareMathOperator{\op}{\textup{op}}

\DeclareMathOperator{\Aut}{\textup{Aut}}
\DeclareMathOperator{\AAAA}{\mathbf{A}}
\DeclareMathOperator{\Spec}{\textup{Spec}}

\newtheorem{theorem}{Theorem}[section]
\newtheorem{lemma}[theorem]{Lemma}
\newtheorem{proposition}[theorem]{Proposition}
\newtheorem{cor}[theorem]{Corollary}
\newtheorem{conj}[theorem]{Conjecture}

\theoremstyle{definition}
\newtheorem{definition}[theorem]{Definition}

\theoremstyle{remark}
\newtheorem{remark}[theorem]{Remark}

\numberwithin{equation}{section}

\makeatletter
\newsavebox\myboxA
\newsavebox\myboxB
\newlength\mylenA

\newcommand*\xoverline[2][0.75]{%
    \sbox{\myboxA}{$\m@th#2$}%
    \setbox\myboxB\null
    \ht\myboxB=\ht\myboxA%
    \dp\myboxB=\dp\myboxA%
    \wd\myboxB=#1\wd\myboxA
    \sbox\myboxB{$\m@th\overline{\copy\myboxB}$}
    \setlength\mylenA{\the\wd\myboxA}
    \addtolength\mylenA{-\the\wd\myboxB}%
    \ifdim\wd\myboxB<\wd\myboxA%
       \rlap{\hskip 0.5\mylenA\usebox\myboxB}{\usebox\myboxA}%
    \else
        \hskip -0.5\mylenA\rlap{\usebox\myboxA}{\hskip 0.5\mylenA\usebox\myboxB}%
    \fi}
\makeatother

\begin{document}

\title{Isotropic cuspidal functions in the Hall algebra of a quiver}

\author{Lucien Hennecart}
\address{Laboratoire de Math\'ematiques d'Orsay, Univ. Paris-Sud, CNRS, Universit\'e Paris-Saclay, 91405 Orsay, France}

\email{lucien.hennecart@math.u-psud.fr}

\date{March 11, 2019}



\begin{abstract}
From the structure of the category of representations of an affine cycle-free quiver, we determine an explicit linear form on the space of regular cuspidal functions over a finite field: its kernel is exactly the space of cuspidal functions. Moreover, we show that any isotropic cuspidal dimension has an affine support. Brought together, this two results give an explicit description of isotropic cuspidal functions of any quiver. The main theorem together with an appropriate action of some permutation group on the Hall algebra provides a new elementary proof of two conjectures of Berenstein and Greenstein previously proved by Deng and Ruan. We also prove a statement giving non-obvious constraints on the support of the comultiplication of a cuspidal regular function allowing us to connect both mentioned conjectures of Berenstein and Greenstein. Our results imply the positivity conjecture of Bozec and Schiffmann concerning absolutely cuspidal polynomials in isotropic dimensions.
\end{abstract}

\maketitle

\tableofcontents

\section{Introduction}
Primitive elements of Hopf algebras or more generally of bialgebras are of primary importance in their study. A striking result is the Milnor-Moore theorem (\cite{MR0174052}) asserting that a graded connected cocommutative Hopf algebra with finite dimensional graded parts is isomorphic to the enveloping algebra of the Lie algebra of its primitive elements. Primitive elements of quantum groups (\cite{LusztigQuantum}) have no mystery: these are the Chevalley generators. The situation for generalized Borcherds-Kac-Moody algebras (\cite{Borcherds, MR1341758}) is analogous: these behave in fact very much like quantum groups associated to Kac-Moody algebras, although an infinite number of generators and imaginary simple roots are allowed. The Hall algebra of a quiver gives a way to construct quantum groups. Namely, given a quiver $Q$, one can consider the category $\Rep_Q(\F_q)$ of finite dimensional representations of $Q$ over some finite field $\F_q$. It can be used to built the so-called Hall algebra of $Q$ over $\F_q$, which is a Hopf algebra object in some braided monoidal category\footnote{More precisely, in the category of $\Z^I$-graded $\C$-vector spaces with finite dimensional graded components with braiding $X\otimes Y\rightarrow Y\otimes X$, $x\otimes y\mapsto \nu^{(x,y)}y\otimes x$ for any objects $X$ and $Y$ and homogeneous $x,y$, where $I$ is the set of vertices of $Q$ and $(-,-)$ is the symmetrized Euler form defined in Section \ref{2}.}. There is a natural subalgebra of the Hall algebra. It is the subalgebra generated by the classes of simple representations $[S_i]$ at each vertex $i$ of $Q$. This is the so-called composition algebra. By a theorem of Ringel (\cite{MR1062796}), it is isomorphic to the positive part of the quantum group $\U_{\nu}(\mathfrak{g}_A)$ specialized at $\nu=q^{1/2}$. The work of Sevenhant and Van den Bergh (\cite{SevenhantVdB}) identifies the whole Hall algebra $\HH_{Q,\F_q}$ of $Q$ with the quantization of the enveloping algebra of a generalized Kac-Moody algebra. The isomorphism is constructed using primitive elements of $\HH_{Q,\F_q}$ and depends on such a choice. This is not completely satisfactory since we would like to determine natural generators of the Hall algebra. This work is the beginning of this project as we provide a way to compute explicitly primitive elements of the Hall algebra in isotropic dimensions. This solves the problem of the calculation of primitive elements of the Hall algebra for affine quivers, but the ambiguity still remains.

Here is a brief overview of what is done in this paper. In Section \ref{2}, we introduce notations and known facts of the representation theory of quivers. We focuse in particular on the category of representations of affine quivers and recall their decomposition in blocks. In Section \ref{3}, we recall the definition of the constructible Hall algebra of a quiver. We provide several formulas for the comultiplication and recall the theorem of Sevenhant and Van den Bergh, which will only be used in Section \ref{7} to prove Conjecture \ref{conj2}. A major role is played by the Kronecker quiver for which the classification of representations is explicit. In Section \ref{4}, we write the formulas for the number of indecomposable and absolutely indecomposable representations, and for the dimensions of cuspidal functions for affine quivers. In Section \ref{5}, we calculate explicitly all cuspidal functions of the Jordan quiver. We do not know any formula for nilpotent cuspidal functions of cyclic quivers but we provide sufficient informations on the value they take on indecomposable representations to deal with them. In Section \ref{6}, we determine regular cuspidal functions of affine quivers. Regular cuspidality is a weaker condition than cuspidality. As a numerical coincidence, in imaginary dimensions, cuspidal functions form a codimension one subspace of regular cuspidal functions. We determine an explicit linear form defining this hyperplane. In Section \ref{7}, we use our results to prove two conjectures made by Berenstein and Greenstein in \cite{MR3463039} concerning the symmetry of the Hall algebra. The last Section \ref{8} is devoted to show that an isotropic cuspidal dimension of any quiver has affine support. This immediately implies a conjecture of Bozec and Schiffmann in isotropic dimensions. The letter $I$ is used for both the set of vertices of a quiver and an indecomposable representation of a quiver. It should be clear from the context how to distinguish them.

\medskip

\subsection{The main results}
We state here our main contributions. Let $Q$ be an affine quiver and $\F_q$ a finite field. As in Section \ref{3}, let $\HH_{Q,\F_q}$ be the Hall algebra of $Q$ over $\F_q$.
\subsubsection{Cuspidal functions as the kernel of a linear form}
In Section \ref{6}, we consider the subalgebra $\HH_{Q,\F_q,\RC}$ of $\HH_{Q,\F_q}$ generated by classes $[M]$ for $M$ a regular representation. It is a Hopf algebra for a corestriction of the comultiplication of $\HH_{Q,\F_q}$\footnote{see the introduction for precisions on the bialgebra structure.} endowed with a nondegenerate hermitian product $(-,-)$. The algebra $\HH_{Q,\F_q}$ has a well-understood structure. Indeed, we have a bialgebra graded isomorphism
\[
 \HH_{Q,\F_q,\RC}\simeq \underset{a\in \lvert\PP^1_{\F_q}\rvert}{\bigotimes^{}{}^{\prime}}\HH_{a}
\]
where $\HH_{a}$ is isomorphic to Macdonald's ring of symmetric function or to the nilpotent Hall algebra of a cyclic quiver for some finite number of $a$, and the degree of elements of $\HH_a$ for $a\in\lvert\PP^1_{\F_q}\rvert$ is multiplied by $\deg(a)$ to obtain the degree in $\HH_{Q,\F_q,\RC}$. This decomposition allows us to give an expression of primitive elements of this algebra, called regular cuspidal functions. For $r\geq 1$ and $\delta$ the indecomposable imaginary root of $Q$, let $\HH_{Q,\F_q,\RC}^{\cusp}[r\delta]$ be the subspace of regular cuspidal functions of dimension $r\delta$. We also let
\[
 \chi_{r\delta}=\sum_{\substack{[M]\text{ regular}\\ \dim M=r\delta}}[M].
\]
\begin{theorem}
The kernel of the linear form
\[
 \begin{matrix}
  L :& \HH_{Q,\F_q,\RC}^{\cusp}&\rightarrow &\C\\
  &f&\mapsto&(f,\chi_{r\delta})
 \end{matrix}
\]
is $\HH_{Q,\F_q}^{\cusp}[r\delta]$.
\end{theorem}
\subsubsection{An action of a permutation group on the Hall algebra}
To prove Conjecture \ref{conj2}, we use the following action of a product of permutation groups on the Hall algebra which deserves to be considered separately.

As in Section \ref{2}, $D$ denote the set of closed points of $\PP^1_{\F_q}$ parametrizing non-homogeneous tubes. For $e\geq 2$, we let $N(e)$ be the number of closed points of $\PP^1_{\F_q}$ of degree $e$ and $N(1)=q+1-|D|$. Let $\mathfrak{S}$ be the group of degree preserving permutations of $\lvert\PP^1_{\F_q}\rvert\setminus D$. The group $\mathfrak{S}$ is isomorphic to
\[
 \prod_{e\geq 1}\mathfrak{S}_{N(e)}
\]
where for a positive integer $N$, $\mathfrak{S}_N$ is the symmetric group on $N$ letters. We define an action
\[
 \mathfrak{S}\rightarrow \Aut(\HH_{Q,\F_q})
\]
as follows. For $M,N$ two representations, $\lambda$ a partition, $x\in\lvert\PP^1_{\F_q}\rvert\setminus D$ and $\sigma\in \mathfrak{S}$,
\begin{gather*}
 \sigma\cdot [M]=[M] \quad \text{ if $[M]$ is preprojective, preinjective or in a non-homogeneous tube}\\
 \sigma\cdot [I_{\lambda}(x)]=[I_{\lambda}(\sigma(x))]\\
 \sigma\cdot ([M]\oplus [N])=\sigma\cdot [M]\oplus \sigma\cdot[N]
\end{gather*}
where for notational reasons, we define here $[M]\oplus[N]=[M\oplus N]$.
It is easily seen that $\sigma$ acts as a graded linear isomorphism on $\HH_{Q,\F_q}$. We prove in Section \ref{7} the following facts.
\begin{enumerate}
 \item $\sigma$ acts as an isometry of $\HH_{Q,\F_q}$,
 \item The action of $\sigma$ leaves $\HH_{Q,\F_q,\RC}$ and $\HH_{Q,\F_q,\RC}^{\cusp}$ stable,
 \item $\sigma$ commutes with the linear form $L$. In particular, it preserves $\HH_{Q,\F_q}^{\cusp}[\dd]$ for any dimension $\dd$.
\end{enumerate}
This yields the following result.

\begin{theorem}
The group $\mathfrak{S}$ acts on $\HH_{Q,\F_q}$ by degree preserving Hopf algebra automorphisms.
\end{theorem}
Usually, quantum groups have very few degree preserving Hopf algebra automorphisms. The only ones are obtained by rescaling the Chevalley generators. Here, there is multiplicities for imaginary roots, from which we obtain non-trivial automorphisms.


\subsection{Acknowledgements} I warmly thank Olivier Schiffmann for his constant support and availability during the preparation of this paper and for many useful comments on a preliminary version.

\section{Structure of the category of representations of affine quivers over a finite field}\label{2}
\subsection{Notation and recollections on quiver representations}
In this section, let $Q=(I,\Omega)$ be an arbitrary quiver with set of vertices $I$ and set of arrows $\Omega$ and $k$ a field.  We denote by $\Rep_Q(k)$ the category of finite dimensional representations of $Q$ over $k$. Equivalently, this is the category of finite dimensional modules over the path algebra $kQ$ of $Q$. This category is known to be a $k$-linear abelian category of homological dimension one.
\subsubsection{The Euler form}For $\mathcal{A}$ an abelian category, $K_0(\mathcal{A})$ is its Grothendieck group. For $M$ an object of $\mathcal{A}$, $\xoverline{M}\in K_{0}(\mathcal{A})$ is its class in the Grothendieck group. The bilinear (usually non-symmetric) form
\[
\langle -,-\rangle : K_0(\mathcal{A})\times K_0(\mathcal{A})\rightarrow \Z,\quad\quad \langle \xoverline{M},\xoverline{N}\rangle=\homm(M,N)-\ext^1(M,N)
\]
where by definition $\homm(M,N)=\dim_k \Hom_{kQ}(M,N)$ and $\ext^1(M,N)=\dim_k\Ext^1_{kQ}(M,N)$ is called the Euler form of the quiver. It factorizes through the (surjective) morphism of abelian groups $\dim : K_0(\mathcal{A})\rightarrow \Z^I$ given by the dimension $\xoverline{M}\mapsto \dim M$. In fact, if $\dd=\dim M\in \N^I$ and $\dd'=\dim N\in \N^I$, we have the explicit formula :
\[
 \langle\xoverline{M},\xoverline{N}\rangle=\sum_{i\in I}\dd_i\dd'_i-\sum_{\alpha : i\rightarrow j}\dd_i\dd'_j.
\]

We use the same notation
\[
 \langle -,-\rangle : \Z^I\times \Z^I\rightarrow \Z
\]
for the induced bilinear form. We will also consider its symmetrized version :
\[
(-,-) : \Z^I\times \Z^I\rightarrow \Z, \quad\quad (\dd,\dd')=\langle\dd,\dd'\rangle+\langle\dd',\dd\rangle \quad \text{for any $\dd,\dd'\in\Z^I$}.
\]
In case of an affine quiver, the symmetrized Euler form is nonnegative with one dimensional kernel generated by an indecomposable integer valued positive vector, called the indecomposable imaginary root, denoted by the letter $\delta$ (see \cite[Theorem 8.6]{SchifflerQuiver}).

\subsubsection{Dualization}\label{dualisation}Let $Q$ be a quiver and $Q^*$ the quiver in which we change the orientation of all arrows, which deserves the name of dual quiver. A representation $V$ of $Q$ gives a representation $V^*$ of $Q^*$ obtained by dualizing the vector spaces at the vertices of $Q$ and replacing the linear maps between them by their transpose. Explicitly, if $V=((V_i)_{i\in I}, (f_{\alpha})_{\alpha\in\Omega})$ is a representation of $Q$, its dual is $V^*=((V_i^*)_{i\in I}, (^t\!f_{\alpha})_{\alpha\in\Omega})$. We obtain in this way an equivalence of categories
\[
\tilde{D} : \Rep_Q(k)^{\op}\rightarrow\Rep_{Q^*}(k)
\]
which associate to a representation its dual. In particular, there is an identification $\Hom_{kQ^*}(M^*,N^*)\simeq \Hom_{kQ}(N,M)$ and $\Ext^1_{kQ^*}(M^*,N^*)\simeq \Ext_{kQ}^1(N,M)$.

\subsection{The structure of the category of representations of an affine quiver}
\subsubsection{Cyclic and Jordan quivers}
Let $k$ be an arbitrary field. Cyclic and Jordan quivers are the non-acyclic quivers of affine type. This section introduces notations which will be later used. Let $J$ be the Jordan quiver. We let $J_n$ be the $n\times n$ indecomposable nilpotent matrix with ones on the superdiagonal and zeros everywhere else:
\[
\begin{pmatrix}
0&1&0&\hdots&0\\
0&\ddots&\ddots& &0\\
\vdots&\ddots&\ddots&\ddots&\vdots\\
&&\ddots&\ddots&1\\
0&\hdots&\hdots&0&0
\end{pmatrix}.
\]
We also see $J_n$ as a $n$-dimensional representation of the Jordan quiver. For a partition $\lambda=(\lambda_1,\hdots)$, we write $J_{\lambda}=\bigoplus_{j\geq 1}J_{\lambda_j}$. Any nilpotent representation of $J$ over $k$ is isomorphic to exactly one representation of the following set:
\[
\{J_{\lambda} : \lambda\text{ partition}\}.
\]
All other isomorphism classes of representations of the Jordan quiver also have an explicit representative which will not be used. See \cite[Section 3.]{HuaCounting}.

Let $n\geq 1$ be a integer and $C_n$ be the cyclic quiver of length $n$. What follows also applies to $C_1$ which is the Jordan quiver. We distinguish two types of representations for cyclic quivers: invertible representations, for which every arrow is invertible, and nilpotent representations, for which the composition of arrows along any sufficiently long path is zero. Before describing them, we introduce some notations. We suppose that the vertices are indexed by $\Z/n\Z$ with exactly one arrow $i\rightarrow i+1$ for $i\in \Z/n\Z$. A representation of $C_n$ is a $n$-tuple $(a_{0},\hdots,a_{n-1})$ where $a_i : V_i\rightarrow V_{i+1}$ for $i\in\Z/n\Z$ and some $k$-vector spaces $V_i$. We let $\Rep_{C_n}^{\inv}(k)$ be the full subcategory of invertible representations of $C_n$ over $k$ and $\Rep_{C_n}^{\nil}(k)$ be the full subcategory of nilpotent representations.
\begin{enumerate}
\item $\Rep_{C_n}(k)=\Rep_{C_n}^{\nil}(k)\sqcup\Rep_{C_n}^{\inv}(k)$ is a decomposition in blocks of $\Rep_{C_n}(k)$.
\item A full set of representatives of nilpotent indecomposable representations is built as follows. Let $l$ be a nonnegative integer and for $0\leq m\leq l$, $V_m=ke_m$ a one-dimensional vector space generated by $e_m$. The $\Z/n\Z$-graded vector space $V_{i,l}=\bigoplus_{\bar{x}\in\Z/n\Z}\bigoplus_{m\equiv x-i\pmod{n}}V_m$ with the endomorphism sending $e_m$ to $e_{m+1}$ if $0\leq m<l$ and $e_l$ to zero defines an indecomposable representation of $C_n$ again denoted by $V_{i,l}$. The set
\[
\{V_{i,l} : i\in\Z/n\Z, l\geq 0\}
\]
contains exactly one representative of each isomorphism class of indecomposable nilpotent representations. Define also $S_i=V_{i,0}$ for $i\in\Z/n\Z$.

\item The following functor is an equivalence of categories:
\begin{equation}\label{equivJC}
\begin{matrix}
G_n :& \Rep_{J}^{\inv}(k)&\rightarrow &\Rep_{C_n}^{\inv}(k)\\
&(V,a)&\mapsto&(\id,\hdots,\id,a).
\end{matrix}
\end{equation}
\end{enumerate}
\begin{proof}The statement $(3)$ is straightforward and $(2)$ is proved in \cite[Section 3.5]{SchiffmannHall}. We prove $(1)$. Since $\Rep_{C_n}^{\nil}(k)$ and $\Rep_{C_n}^{\inv}(k)$ are stable under taking subobjects and quotients in $\Rep_{C_n}(k)$, any morphism from a nilpotent representation $M$ to an invertible representation $N$ is zero, and conversely. The extension spaces $\Ext^{1}(M,N)$ and $\Ext^{1}(N,M)$ also vanish, since using the Euler form and that $\dim N$ is some multiple $r\delta=(r,\hdots,r)$ of the indecomposable imaginary root $\delta=(1,\hdots,1)$, we have:
\[
0=(\xoverline{M},\xoverline{N})=-\ext^1(M,N)-\ext^1(N,M)=0.
\]

\end{proof}

\subsubsection{Decomposition in blocks of $\Rep_{C_n}(k)$}\label{blockCn}
Let $k$ be a field. We give here a decomposition of $\Rep_{C_n}(k)$ in blocks. Let $V$ be a $k$-vector space of dimension $d$. The group $\GL(V)$ acts algebraically on $\Hom(V,V)$ by conjugation with quotient $\Hom(V,V)//\GL(V)\simeq S^n\AAAA^1_k$. Denote by $\pi_V$ the projection $\pi_V : \Hom(V,V)\rightarrow S^n\AAAA^1_k$.
\begin{theorem}
We have a decomposition in blocks
\[
\Rep_{J}(k)\simeq \bigsqcup_{a\in\mid\AAAA^1_k\mid}\Rep^a_J(k),
\]
where $(V,x)\in\Rep^a_J(k)$ if and only if $\pi_V(x)=(a,\hdots,a)$.
\end{theorem}
\begin{theorem}
We have a decomposition in blocks
\[
\Rep_{C_n}(k)\simeq\bigsqcup_{a\in\mid\AAAA^1_k\mid}\Rep_{C_n}^a(k)
\]
where $\Rep_{C_n}^0(k)=\Rep_{C_n}^{\nil}(k)$ and $\Rep_{C_n}^a(k)$ is the full subcategory $G_n(\Rep_J^{a}(k))$ of $\Rep_{C_n}^{\inv}(k)$.
\end{theorem}

\subsubsection{Acyclic affine quivers}
We recollect known facts on the representation theory of acyclic affine quivers (see \cite{MR774589}) over a finite field (some results may hold in greater generality). The exposition here follows and can be completed by \cite[Section 3.6]{SchiffmannHall} and \cite[\S 8]{CBlec}. Troughout this section, $Q=(I,\Omega)$ is an acyclic affine quiver. This condition is equivalent to the finite dimensionality of the path algebra of $Q$ over any field and excludes the Jordan quiver and cyclic quivers which have been studied above.

\begin{theorem}
Let $k$ be an arbitrary field. Then, there exists an adjunction
\[
\tau^- : \Rep_{Q}(k) \rightleftarrows \Rep_Q(k) : \tau
\]
with bi-natural isomorphisms\footnote{We say that $(\tau^-,\tau)$ is a Serre adjunction.} (the star means the dual with respect to the $k$-vector space structure):
\[
\Ext^1(M,N)^*\simeq \Hom(N,\tau M),\quad\quad \Ext^1(M,N)^*\simeq \Hom(\tau^-N,M).
\]
\end{theorem}
The functors $\tau$ and $\tau^-$ are known as \emph{Auslander-Reiten translates}. From the properties of $\tau^-$ and $\tau$, it is immediate that a representation $M$ of $Q$ over $k$ is projective if and only if $\tau(M)=0$ and injective if and only if $\tau^-(M)=0$. We call an indecomposable representation $M$ of $Q$ over $k$
\begin{enumerate}
\item preprojective if $\tau^nM=0$ for $n\gg 0$,
\item preinjective if $\tau^{-n}M=0$ for $n\gg0$,
\item regular if $\tau^nM\neq 0$ for all $n\in\Z$.
\end{enumerate}
Furthermore, we call a representation $M$ of $Q$ over $k$ preprojective if all its indecomposable direct summands are preprojective, and we adopt similar terminology for preinjective and regular representations. The full subcategory of $\Rep_Q(k)$ of preprojective (resp. preinjective, resp. regular) representations is denoted by $\Rep_Q^{\PC}(k)$ (resp. $\Rep_Q^{\IC}(k)$, resp. $\Rep_Q^{\RC}$). These are extension closed subcategories of $\Rep_Q(k)$, hence exact categories. Moreover, $\Rep_Q^{\RC}(k)$ is an abelian category (though not stable under taking subobjects in the bigger category $\Rep_k Q$). The three categories $\Rep_Q^{\RC}(k), \Rep_Q^{\PC}(k)$ and $\Rep_Q^{\IC}(k)$ are disjoint and the category to which an indecomposable $M$ belongs is given by the sign of its defect defined by $\partial M=\langle\delta,\dim M\rangle$. An indecomposable representation $M$ is preprojective if and only if $\partial M<0$, preinjective if and only of $\partial M>0$ and regular if and only if $\partial M=0$.
The following proposition gives the interactions between these three subcategories.
\begin{proposition}\label{extensions}
For $M\in \Rep_Q^{\PC}(k)$, $N\in\Rep_Q^{\IC}(k)$, $L\in\Rep_Q^{\RC}(k)$, we have
\[
\Hom(N,M)=\Hom(N,L)=\Hom(L,M)=0,
\]
\[
\Ext^1(M,N)=\Ext^1(L,N)=\Ext^1(M,L)=0.
\]
\end{proposition}
The simple objects of the abelian category $\Rep_Q^{\RC}(k)$ are called simple regular. A simple regular representation $M$ is called homogeneous if $\tau M\simeq M$.

\begin{theorem}[Ringel, \cite{MR774589}]\label{ringelth}
Let $Q$ be an affine acyclic quiver and $k$ an arbitrary field. Let $d$ and $p_1,\hdots,p_d$ be attached to $Q$ as in the table below. Then
\begin{enumerate}
\item There is a degree preserving bijection $M_a\leftrightarrow a$ between the set of homogeneous regular simple modules and $\mid\PP^1_k\mid\setminus D$ where $D$ consists of $d$ closed points of degree one\footnote{in the sequel for $X$ a scheme, we denote by $\mid X\mid$ the set of its closed points.},
\item There are $d$ $\tau$-orbits $\OO_1,\hdots,\OO_d$ of non-homogeneous regular simple modules of size $p_1,\hdots,p_d$\footnote{\emph{i.e.} the set of isomorphism classes of simple objects in $\OO_j$, $1\leq j\leq d$ is of cardinality $d$ and the Auslander-Reiten translates $\tau$ and $\tau^-$ act as inverse cycles on it.},
\item The category $\Rep_Q^{\RC}(k)$ decomposes as a direct sum of blocks\footnote{There are no morphisms or extensions between the objects of different categories}:
\[
\Rep_Q^{\RC}(k)=\bigsqcup_{a\in\mid\PP^1\mid\setminus D}\CC_{M_a}\sqcup\bigsqcup_{l=1}^d\CC_{\OO_l}
\]
where $\CC_{M_a}$ is the full subcategory of objects which are extensions of $M_a$ and $\CC_{\OO}$ is the full subcategory of $\Rep_Q^{\RC}(k)$ of objects whose regular simple factors lie in $\OO$.
\end{enumerate}

\begin{figure}[h!]
\begin{tabular}{|c|c|c|}
\hline
\text{type of } $Q$ &$d$&$p_1,\hdots,p_d$\\
\hline
$A_1^{(1)}$&$0$& \\
\hline
$A_n^{(1)}, n>1$&$2$& $p_1=$\text{number of arrows going clockwise}\\
                        &  & $p_2=$\text{number of arrows going counterclockwise}\\
                        \hline
$D_n^{(1)}$&$3$&$2, 2, n-2$\\
\hline
$E_{n}^{(1)}, n=6,7,8$&$3$&$2,3,n-3$\\
\hline
\end{tabular}
\caption{Non-homogeneous tubes of affine quivers and their period \cite[(3.18)]{SchiffmannHall}}
\label{tubes}
\end{figure}
\end{theorem}
In Theorem \ref{ringelth}, the subcategories $\CC_{\OO_l}$ are called the non-homogeneous tubes while the subcategories $\CC_{M_a}$ are the homogeneous tubes. The number of non-homogeneous tubes is $d$ (see however Remark \ref{rem}) and the integers $p_1,\hdots,p_d$ are the periods. They do not depend on the chosen field. For $a\in\mid\PP_Q^1\mid$, $\CC_{Q}^a$ also denotes the corresponding tube.
\begin{remark}\label{rem}
In type $A_n^{(1)}$, in the case where all arrows except one go in the same direction, we have in fact $d=1$, \emph{i.e.} there is only one non-homogeneous tube.
\end{remark}

We can furthermore precisely identify the tubes $\CC_{M_a}$ and $\CC_{\OO}$ with the help of the Jordan quiver and of cyclic quivers respectively.

\begin{theorem}
Let $a\in \mid\PP^1_k\mid\setminus D$ a closed point of degree $d$. Let $K=\End(M_a)$ the $k$-algebra of endomorphisms of the  simple regular $M_a$ of the tube $\CC_{M_a}$. This is a finite field extension of $k$ of degree $d$. There exists a unique equivalence of categories
\[
\begin{matrix}
F_a : &\Rep_J(K)&\rightarrow &\CC_{M_a}\\
&I_{(1)}&\rightarrow &M_a.
\end{matrix}
\]
We set $I_{\lambda}^Q(a):=F_a(J_{\lambda})$.

\medskip

Let $a\in D$ be a closed point corresponding to a non-homogeneous tube. Let $p$ be the corresponding period of the non-homogeneous tube $\OO_a$ and $S$ a simple regular of $\OO_a$. There is a unique equivalence of categories
\[
\begin{matrix}
F_a : &\Rep_{C_p}(k)&\rightarrow&\CC_{\OO_a}\\
&S_i&\rightarrow&\tau^{i}S.
\end{matrix}
\]
\end{theorem}

In the case of a non-homogeneous tube, there is a reasonable way to choose the simple representation $S$. For this, we may first choose an extending vertex $i_0$ of $Q$. Then, since $\sum_{s=0}^p\dim\tau^sS=\delta$ and $\delta_{i_0}=1$, there is a unique simple representation $S$ in the non-homogeneous tube which is nonzero at vertex $i_0$. The isomorphism class of the representation $S$ may however depend on the extending vertex we choose.

\subsection{Identification of the tubes with the help of the Kronecker quiver}
\subsubsection{Representations of the Kronecker quiver}\label{Kronecker}
We recall here the classification of indecomposables of the Kronecker quiver over a finite field $\F_q$ for the sake of completeness. It can be obtained from the classification over the algebraic closure $\xoverline{\F}_q$ using Galois arguments. For complements, see \cite[Section 7.7]{KirQuiv}.

The Kronecker quiver has two vertices connected by two arrows going in the same direction :
\[
K_2 : \quad\begin{tikzcd}
1 \arrow[r,shift left,"\alpha"] \arrow[r,shift right,swap,"\beta"] & 2,
\end{tikzcd}
\]
We fix a finite field with $q$ elements $\F_q$ and an algebraic closure $\xoverline{\F}_q$ of $\F_q$.

\begin{theorem}[{\cite[Theorem 7.30]{KirQuiv}}]\ 
\begin{enumerate}
\item The set of real roots is $\{(n,n+1) : n\in \N\}\cup \{(n+1,n) : n\in\N\}$. For $n\in \N$, the indecomposable representation of $K_2$ over $\F_q$ of dimension $(n,n+1)$ is preprojective whereas the indecomposable representation of $K_2$ over $\F_q$ of dimension $(n+1,n)$ is preinjective. They have the following form:
\[
 \begin{tikzcd}
\F_q^n \arrow[r, shift left,"{\begin{pmatrix}
                          I_n\\
                          0
                         \end{pmatrix}
}"] \arrow[r,shift right,swap,"\begin{pmatrix}
                              0\\
                              I_n
                             \end{pmatrix}
"] & \F_q^{n+1},
\end{tikzcd}
\]
for the preprojective representations and
\[
 \begin{tikzcd}
\F_q^{n+1} \arrow[r,shift left,"\begin{pmatrix}
                          I_n \quad0                
                         \end{pmatrix}
"] \arrow[r,shift right,swap,"\begin{pmatrix}
                              0\quad
                              I_n
                             \end{pmatrix}
"] & \F_q^{n},
\end{tikzcd}
\]
for the preinjective representations, where $I_n$ is the $n\times n$ identity matrix.

\item The regular indecomposable representations of $K_2$ over $\F_q$ are parametrized by the set $\lvert\PP^1_{\F_q}\rvert\times \Z_{\geq 1}$. For $([x:y],n)\in \PP^1_{\F_q}(\xoverline{\F}_q)\times \Z_{\geq 1}$, the corresponding representation of $K_2$ is
\[
 \begin{tikzcd}
\F_{q^{\deg(x)}}^{n} \arrow[r,shift left,"\begin{matrix}
                          xI_n+J_n                
                         \end{matrix}
"] \arrow[r,shift right,swap,"\begin{matrix}
                              y I_n+J_n
                             \end{matrix}
"] & \F_{q^{\deg(x)}}^{n},
\end{tikzcd}
\]
where we consider $\F_{q^{\deg(x)}}$ as a $\F_q$-vector space of dimension $\deg(x)$. The isomorphism class of this representation only depends on the Galois orbit of $[x:y]$. We denote by $I_{t,n}$ where $t=[x:y]$ this representation and when $n=1$, we use the notation $S_t=I_{t,1}$ for $t\in\PP^1_{\F_q}(\xoverline{\F}_q)$. 
\end{enumerate}
\end{theorem}

For $t\in\PP^1_{\F_q}(\xoverline{\F}_q)$ and $\lambda=(\lambda_1,\lambda_2,\hdots)$ a partition, define $I_{t,\lambda}=\oplus_{i=1}^{l(\lambda)}I_{t,\lambda_i}$. The isomorphism class $[I_{t,\lambda}]$ only depends on the image $a$ of $t : \Spec\xoverline{\F}_q\rightarrow \PP^1_{\F_q}$. For any $a\in\lvert\PP^1_{\F_q}\rvert$, we fix $t_a\in\PP^1_{\F_q}(\xoverline{\F}_q)$ a geometric point such that the image of $t : \Spec\xoverline{\F}_q\rightarrow \PP^1_{\F_q}$ is $a$. The set of isomorphism classes of regular representations of $K_2$ over $\F_q$ is
\[
\M^{K_2}_{\F_q}=\{[I_{t_a,\lambda}] : a\in\lvert\PP^1_{\F_q}\rvert, \lambda \text{ a partition}\}.
\]
For $a\in\lvert\PP^1_{\F_q}\rvert$, write $I_{a,\lambda}=I_{t_a,\lambda}$ and $S_a=I_{a,(1)}$. For any $r\geq 1$, the number $I_{K_2,(r,r)}(q)$ of indecomposable representations of the Kronecker quiver of dimension $(r,r)$ can be determined from this:
\[
 I_{K_2,(r,r)}(q)=\text{number of closed points of $\PP^1_{\F_q}$ of degree dividing $r$}.
\]

This description also gives a natural way to identify the tubes of the Kronecker quiver :
\begin{equation}\label{Ktubes}
\begin{matrix}
C : &\lvert\PP^1_{\F_q}\rvert&\rightarrow &\{\text{tubes of $K_2$ over $\F_q$}\}\\
     &a&\mapsto&C_a=\{I_{t_a,\lambda} : \lambda \text{ partition}\}
\end{matrix}.
\end{equation}

\subsubsection{The number of automorphisms of regular representations of the Kronecker quiver}
For $q$ a power of a prime and $\lambda=(1^{l_1},2^{l_2},\hdots)$ a partition, define:
\[
 a_{\lambda}(q)=q^{|\lambda|+2n(\lambda)}\prod_i\prod_{j=1}^{l_i}(1-q^{-j}).
\]
This is the number of automorphism of the nilpotent representation of type $\lambda$ $J_{\lambda}$ of the Jordan quiver over $\F_q$ (see \cite[Lemma 2.8]{SchiffmannHall} or \cite[Chapter II, (1.6)]{macdonald}). The identification of the tubes of the Kronecker quiver with nilpotent representations of the Jordan quiver gives the following.

\begin{proposition}\label{automorphisms}
 Let $a\in\lvert\PP^1_{\F_q}\rvert$ and $n \geq 1$. Then the number of automorphism of $I^{K_2}_{t_a,\lambda}$ is $a_{\lambda}(q^{\deg(a)})$.
\end{proposition}

\subsubsection{Regular representations of an affine quiver}

Let $Q$ be an arbitrary acyclic affine quiver. Ler $i_0$ be an extending vertex of $Q$ and $\delta$ the indivisible imaginary root of $Q$. We can restrict ourselves to the case where $i_0$ is a sink, since it is always possible to choose the extending vertex to be a sink or a source, and the dualization is a way to pass from the second case to the first. Without referring to the dualization process, it is possible to adapt this section in the case where the extending vertex $i_0$ is a source. By the classification above, all tubes are homogeneous for the Kronecker quiver $K_2$, \emph{i.e.} in type $A_1^{(1)}$ with the noncyclic orientation.
\medskip

Let $Q'=(I',\Omega')$ be the finite type quiver associated to $Q$ by erasing the vertex $i_0$ and all arrows adjacent to it. The element $\theta=\delta-e_{i_0}$ is an element of $\N^{I'}$. We denote by $I_{\theta}$ an indecomposable representation of $Q$ of dimension $\theta$. Thanks to the classification of affine quivers there is either one vertex $i_1$ adjacent to $i_0$ (in types $D_n^{(1)}$ ($n\geq 4$) and $E_n^{(1)}, n=6,7,8$), in which case $\delta_{i_1}=2$, or two vertices $i_1, i_2$ adjacent to $i_0$ (in types $A_n^{(1)}$, $n>1$), in which case, $\delta_{i_1}=\delta_{i_2}=1$. In the first case, we choose an arbitrary identification $I_{i_1}\simeq k^2$. We have a functor
\begin{equation}\label{FunctorF}
\begin{matrix}
F : &\Rep_{K_2}(k)&\rightarrow& \Rep_{Q}(k)\\
     &(V_0,V_1,\alpha,\beta)&\mapsto&V
\end{matrix}
\end{equation}
where $V$ is defined as follows. The restriction to the subquiver $Q'$ is $V_{Q'}=I_{\theta}\otimes V_0$, $V_{i_0}=V_1$, and if $i_0$ is connected by two arrows $i_1\rightarrow i_0$ and $i_2\rightarrow i_0$ to $Q'$, then we choose $\alpha$ for the map $V_{i_1}=V_0\rightarrow V_{i_0}=V_1$ and $\beta$ for the map $V_{i_2}=V_0\rightarrow V_{i_0}=V_1$. If $i_0$ is connected to $Q'$ by a single arrow $i_1\rightarrow i_0$, then the map $V_{i_1}\simeq V_0\oplus V_0\rightarrow V_{i_0}$ is choosen to be $\alpha\oplus \beta$. Of course this functor depends in the first case on the choosen order $i_1, i_2$ of the vertices connected to $i_0$ and in the second case on the identification $V_{i_1}\simeq k^2$. We implicitly fix such a choice. The action of $F$ on the morphisms is as follows. Let $V=(V_0,V_1,\alpha,\beta)$ and $V'=(V'_0,V'_1,\alpha',\beta')$ be two representations of $K_2$ and $f : V\rightarrow V'$ a morphism between them. The linear map $f_0 : V_0\rightarrow V'_0$ induces a morphism of representations of $Q'$, $\id\otimes f_0 : I_{\theta}\otimes V_0\rightarrow I_{\theta}\otimes V'_0$ and considering $f_1 : V_1=F(V)_{i_0}\rightarrow V'_1=F(V')_{i_0}$ at the vertex $i_0$, we obtain a morphism of representations $F(f) : F(V)\rightarrow F(V')$. Note that $F(S_1)=I_{\theta}$ and $F(S_2)=S_{i_0}$.
\begin{proposition}[{\cite[Theorem 7.34]{KirQuiv}}]
The functor $F$ is exact and fully faithful.
\end{proposition}
\begin{proof}
For quiver representations, exactness of a sequence can be checked pointwise (\emph{i.e.} at each vertex). Here, tensor products are over a field so exactness of $F$ is immediate from its definition.
\medskip

For the full faithfulness, we use that $I_{\theta}$, being an indecomposable of $Q'$ which is a finite type quiver, is a brick, \emph{i.e.} $\End_{kQ'}(I_{\theta},I_{\theta})\simeq k$. Let $V=(V_0,V_1,\alpha,\beta)$ and $V'=(V'_0,V'_1,\alpha',\beta')$ be two representations of $K_2$ and $g : F(V)\rightarrow F(V')$ a morphism of representations. We obtain for free a linear map $f_1=g_{i_0} : V_1\rightarrow V'_1$. Because $I_{\theta}$ is a brick, the restricted morphism of representations $g_{Q'} : I_{\theta}\otimes V_0\rightarrow I_{\theta}\otimes V'_0$ is induced by a unique linear map $f_1 : V_0\rightarrow V'_0$. The datum $(f_0,f_1) : V\rightarrow V'$ is a morphism of representations such that $F(f_0,f_1)=g$. Faithfulness is immediate. This concludes the proof.
\end{proof}

\begin{proposition}\label{ess}
 The image of the functor $F$ is the full subcategory of $\Rep_{Q}(\F_q)$ whose objects are extensions of $I_{\theta}^{\oplus d_1}$ by $S_{i_0}^{\oplus d_2}$ for some $d_1, d_2\geq 0$.
\end{proposition}
\begin{proof}
 This is a straightforward consequence of the definition of $F$. By definition, if $V=(V_0,V_1,\alpha,\beta)$ is a representation of $K_2$, then $F(V)$ is an extension of $I_{\theta}\otimes V_0$ by $S_{i_0}\otimes V_1$. Conversely let $(V', (f_{\alpha})_{\alpha\in\Omega})$ be a representation of $Q$, extension of $I_{\theta}^{\oplus d_1}$ by $S_{i_0}^{\oplus d_2}$ for some $d_1,d_2\geq 0$. We argue separately according to whether $Q=A_n^{(1)}$ or $Q \in\{D_n^{(1)}, E_n^{(1)}\}$. In the first case, we define the representation $V=(V_0,V_1,\alpha,\beta)$ of $K_2$ by setting $V_1=V'_{i_0}$, $V_0=V'_{i_1}$. An isomorphism of the restriction of $V'$ to $Q'$ with $I_{\theta}^{\oplus d_1}$, $\psi : V'_{Q'}\rightarrow I_{\theta}^{\oplus d_1}$ induces an isomorphism $\psi : V'_{i_1}\rightarrow V'_{i_2}$. We define $\alpha = f_{i_1\rightarrow i_0}$ and $\beta = f_{i_2\rightarrow i_0}\circ \psi$. It is easily checked that the isomorphism class of the so defined representation $V$ does not depend on the various choices made and that $F(V)\simeq V'$.
\medskip

In the second case, an isomorphism $V'/S_{i_0}^{d_2}\simeq I_{\theta}^{\oplus d_1}$ gives an identification $\varphi : V'_{i_1}\rightarrow k^2\otimes k^{d_1}$. We then define $V$ by setting $V_0=k^{d_1}$, $V_1=V_{i_0}$ and we define $\alpha=f_{i_0\rightarrow i_2}\circ\varphi^{-1}((1,0)\otimes -)$ and $\beta =f_{i_0\rightarrow i_2}\circ\varphi^{-1}((0,1)\otimes -)$. This defines a representation $V=(V_0,V_1,\alpha,\beta)$ of $K_2$ such that $F(V)\simeq V'$.
\end{proof}

Following the notations of Kirillov in \cite{KirQuiv}, we define $S_a^{Q}=F(S_a)$ for $a\in \lvert\PP^1_{\F_q}\rvert$.
\begin{theorem}[Identification of the simple regular representations, {\cite[Theorem 7.37]{KirQuiv}}]\label{identif}\ 

\begin{enumerate}
\item Let $X$ be a simple regular representation of $Q$ over $\F_q$ of period $1$. Then $\dim X=d\delta$ where $d$ is the degree of the tube containing $X$ and $X\simeq S_a^Q$ for some closed point $a\in\lvert\PP^1_{\F_q}\rvert$.
\item Let $\OO$ be the $\tau$-orbit of a simple regular representation of period $l>1$. Then
\[
\sum_{X\in\OO}\dim X=\delta.
\]
Since at the extending vertex $i_0$, $\delta_{i_0}=1$, $\OO$ contains a unique simple regular representation $X$ such that $X_{i_0}\neq 0$. Moreover, if $X^{(l)}\in\CC_{\OO}$ is the unique indecomposable representation of dimension $\delta$ having $X$ as quotient, then $X^{(l)}\simeq S_a^Q$ for some $a\in\lvert\PP^1_{\F_q}\rvert$.
\end{enumerate}
\end{theorem}

\subsection{Some examples}
We treat here the examples of types $D_4^{(1)}$ and $A_4^{(1)}$ with some particular orientations in dimension $\delta$ to illustrate the previous procedure in the two different cases for which (1) the extending vertex is connected to one vertex and (2) the extending vertex is connected to two vertices.
\subsubsection{The type $D_4^{(1)}$} Let $Q$ be the quiver
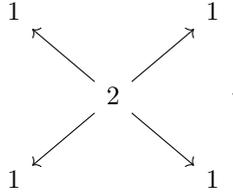
\[
\begin{tikzcd}
1\arrow[leftarrow,rd]                                                                                             &                           & 1\\
&2 \arrow[rightarrow,rd] \arrow[ru]&\\
1\arrow[leftarrow,ru]                                                                                                  &                     &1
\end{tikzcd}.
\]
The labels of the vertices are given by the indecomposable imaginary root $\delta$ of $Q$. The extending vertex can be one of the four vertices with a $1$. To fix the notations, we choose the bottom right vertex and call it $i_0$ as before. An indecomposable representation of $Q$ of dimension $\theta=\delta-e_{i_0}$ is 
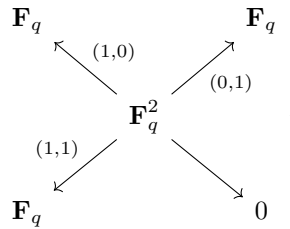
\[
\begin{tikzcd}
\F_q\arrow[leftarrow,rd]{}{(1,0)}                                                                                        &                           & \F_q\\
&\F_q^2 \arrow[rightarrow,rd]{}{} \arrow[ru,swap]{}{(0,1)}&\\
\F_q\arrow[leftarrow,ru]{}{(1,1)}                                                                                               &                     &0
\end{tikzcd}.
\]
Up to isomorphism, it is unique. Then, isomorphism classes of $(1,1)$-dimensional representations of $K_2$ are parametrized by $\PP^1_{\F_q}(\F_q)$. The functor $F$ induces on isomorphism classes of regular representations of dimension $\delta$ the following map :
\[
\begin{matrix}
F : &\M_{\F_q}^{K_2}[\delta]\simeq \PP^1_{\F_q}(\F_q)&\rightarrow &\M_{\F_q}^Q[\delta]\\
     &[\lambda : \mu]&\mapsto& [S_{(\lambda, \mu)}^Q]
\end{matrix}
\]
where $S_{(\lambda, \mu)}^Q$ is the following representation of Q :
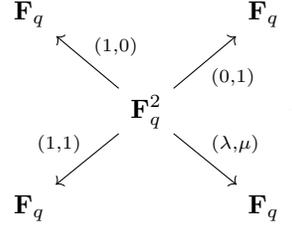
\[
\begin{tikzcd}
\F_q\arrow[leftarrow,rd]{}{(1,0)}                                                                                        &                           & \F_q\\
&\F_q^2 \arrow[rightarrow,rd]{}{(\lambda,\mu)} \arrow[ru,swap]{}{(0,1)}&\\
\F_q\arrow[leftarrow,ru]{}{(1,1)}                                                                                               &                     &\F_q
\end{tikzcd}.
\]
There are three non-homogeneous tubes associated to the parameters $[1:0], [0:1], [1:1]\in\PP^1_{\F_q}(\F_q)$. The other parameters give simple regular representations of dimension $\delta$. We now study one non-homogeneous tube. By symmetry of the quiver, this suffices to determine the three non-homogeneous tubes. Let us choose the non-homogeneous tube associated to $[1:0]$. Thanks to Theorem \ref{identif}, the regular simple representations of the tube are the indecomposables
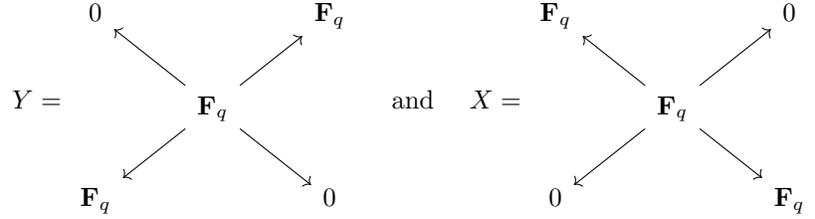
\[
Y=\begin{tikzcd}
0\arrow[leftarrow,rd]{}{}                                                                                        &                           & \F_q\\
&\F_q \arrow[rightarrow,rd]{}{} \arrow[ru,swap]{}{}&\\
\F_q\arrow[leftarrow,ru]{}{}                                                                                               &                     &0
\end{tikzcd}
\quad
\text{ and }
\quad
X=\begin{tikzcd}
\F_q\arrow[leftarrow,rd]{}{}                                                                                        &                           & 0\\
&\F_q \arrow[rightarrow,rd]{}{} \arrow[ru,swap]{}{}&\\
0\arrow[leftarrow,ru]{}{}                                                                                               &                     &\F_q
\end{tikzcd}.
\]
In this non-homogeneous tube, there are two indecomposables of dimension $\delta$ up to isomorphism: the representation $S_{(1,0)}^Q$ and the unique (up to isomorphism) extension of $Y$ by $X$:
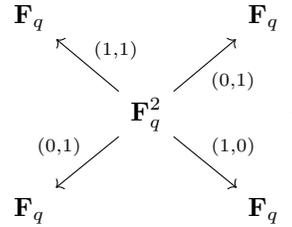
\[
\begin{tikzcd}
\F_q\arrow[leftarrow,rd]{}{(1,1)}                                                                                        &                           & \F_q\\
&\F_q^2 \arrow[rightarrow,rd]{}{(1,0)} \arrow[ru,swap]{}{(0,1)}&\\
\F_q\arrow[leftarrow,ru]{}{(0,1)}                                                                                               &                     &\F_q
\end{tikzcd}.
\]

\subsubsection{The type $A_4^{(1)}$} Let $Q$ be the quiver
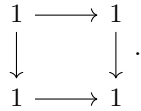
\[
\begin{tikzcd}
1\arrow[r] \arrow[d]&1\arrow[d]\\
1\arrow[r]&1
\end{tikzcd}.
\]
The extending vertex is the bottom right vertex, we call it $i_0$ and the indivisible imaginary root $\delta$ is given in the graph above. The indecomposable representation of dimension $\theta=\delta-e_{i_0}$ is
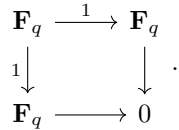
\[
\begin{tikzcd}
\F_q\arrow[r]{}{1} \arrow[d,swap]{}{1}&\F_q\arrow[d]\\
\F_q\arrow[r]& 0
\end{tikzcd}.
\]

The functor $F$ induces on isomorphism classes the following map

\[
\begin{matrix}
F : &\M_{\F_q}^{K_2}[(1,1)]\simeq \PP^1_{\F_q}(\F_q)&\rightarrow &\M_{\F_q}^Q[\delta]\\
     &[\lambda : \mu]&\mapsto& [S_{(\lambda, \mu)}^Q]
\end{matrix}
\]
where $S_{(\lambda, \mu)}^Q$ is the following representation of $Q$:
\[
\begin{tikzcd}
\F_q\arrow[r]{}{1} \arrow[d,swap]{}{1}&\F_q\arrow[d]{}{\mu}\\
\F_q\arrow[r,swap]{}{\lambda}& \F_q
\end{tikzcd}.
\]
The two non-homogeneous tubes correspond to the parameters $[\lambda:\mu]=[0:1]$ and $[\lambda:\mu]=[1:0]$ and both are of period two. Because of the symmetry of the quiver, we study only the case $[\lambda:\mu]=[1:0]$. The two regular simple representations of this tube are
\[
Y=\begin{tikzcd}
0\arrow[r]{}{} \arrow[d]{}{}&\F_q\arrow[d]{}{}\\
0\arrow[r]{}{}& 0
\end{tikzcd}
\quad
\text{ and }
\quad
X=\begin{tikzcd}
\F_q\arrow[r,swap]{}{} \arrow[d,swap]{}{1}&0\arrow[d]{}{}\\
\F_q\arrow[r,swap]{}{1}& \F_q
\end{tikzcd}.
\]
The two indecomposable of dimension $\delta$ in this tube are
\[
\begin{tikzcd}
\F_q\arrow[r]{}{1} \arrow[d,swap]{}{1}&\F_q\arrow[d]{}{0}\\
\F_q\arrow[r,swap]{}{1}& \F_q
\end{tikzcd}
\quad
\text{and}
\quad
\begin{tikzcd}
\F_q\arrow[r]{}{1} \arrow[d,swap]{}{1}&\F_q\arrow[d]{}{1}\\
\F_q\arrow[r,swap]{}{0}& \F_q
\end{tikzcd}.
\]

\subsubsection{Convention}In the sequel, we assume that a functor as in \eqref{FunctorF}:
\[
 F : \Rep_{K_2}(\F_q)\rightarrow \Rep_{Q}(\F_q)
\]
is fixed. The bijection \eqref{Ktubes} gives an explicit bijection
\[
\begin{matrix}
 C_Q : &\lvert\PP^1_{\F_q}\rvert&\rightarrow &\{\text{tubes of $Q$}\}\\
 &a&\mapsto&\text{tube of Q containing $F(C_a)$}.
 \end{matrix}
\]
where $F(C_a)$ is the essential image of the tube $C_a$ of the Kronecker quiver by the functor $F$. We will sometimes also write $C_{Q}^a$ the $a$-tube of $Q$.

\subsubsection{The regular indecomposable representations of an affine quiver}
In the previous theorem, we identified the simple regular representations in the homogeneous tubes. We give now an easy consequence concerning indecomposables in the tubes. We keep the notations of the previous sections.

\begin{proposition}\label{indtubes}
Let $a\in\lvert\PP^1_{\F_q}\rvert$ be a closed point. If the $a$-tube of $Q$ is homogeneous, then for $n\geq 1$, $F(I_{a,n})$ is the unique $n\deg(x)\delta$-dimensional indecomposable representation of this tube.

Suppose the $a$-tube is non-homogeneous. Let $X_a$ be the simple regular representation of $Q$ in this tube which is nonzero at the extending vertex. Then $F(I_{a,n})$ is the indecomposable representation of this tube of dimension $n\delta$ having $X_a$ as quotient.
\end{proposition}
\begin{proof}
This result is a consequence of the full faithfulness of $F$ and the properties of the representations of a tube given by the identification of a tube with the category of nilpotent representations of the Jordan or cyclic quiver.
\end{proof}
Let $a\in \lvert\PP^1_{\F_q}\rvert$ corresponding to a non-homogeneous tube $\OO_a$ and $S_a=F(I_{a,1})$ as in the Proposition \ref{indtubes}. Then the subcategory of $\OO_a$ generated by $S_a$ is isomorphic to the category of nilpotent representations of the Jordan quiver. When $a$ is associated to a homogeneous tube, the restriction of $F$ induces an equivalence of categories between the corresponding tubes of $K_2$ and $Q$.

\section{The Hall algebra of a quiver}\label{3}
The Hall algebra of a quiver now has a long history beginning with the paper of Ringel, \cite{MR1062796}.
\subsection{Definition of the Hall algebra of a quiver}
We refer the reader to \cite{SchiffmannHall} for the proofs and the general theory of constructible Hall algebras. These can be defined for any abelian finitary category. We will consider it for the category $\Rep_Q(k)$ with $k=\F_q$. Let $Q$ be a finite quiver and $k=\F_q$ a finite field. As a vector space, the Hall algebra of $Q$ over $k$ is defined as
\[
\HH_{Q,k}=\bigoplus_{[M]\in \Ob(\Rep_Q(k))/\sim}\C[M]
\]
where $\Ob(\Rep_Q(k))/\sim$ is the set of representations of $Q$ over $k$ up to isomorphism.

It is convenient to see the Hall algebra as an algebra of functions endowed with some sort of convolution product. For this, let $\M_k=\Rep_Q(k)/\sim$ be the set of isomorphism classes of representations of $Q$ over $k$ and
\[
\M_k=\bigsqcup_{\dd\in\N^I}\M_k[\dd]
\]
its decomposition with respect to the dimension. Thus, 
\[
\HH_{Q,k}=\bigoplus_{\dd\in\N^I}\HH_{Q,k}[\dd],\quad\quad \HH_{Q,k}[\dd]=\Fun(\M_k[\dd],\C)=\text{functions $\M_k[\dd]\rightarrow\C$}.
\]
We now define the operations. Let $\nu=|k|^{1/2}$ and for $M$ a representation of $Q$, $a_M=|\Aut(M)|$.
\begin{enumerate}
\item(Multiplication) For $f, g\in \HH_{Q,k}$,
\[
f\star g=m(f,g) : [M]\mapsto \sum_{N\subseteq M}\nu^{\langle M/N,N\rangle}f([M/N])g([N])
\]
where the sum is over subrepresentations $N$ of $M$.
\item(Comultiplication) For $f\in \HH_{Q,k}$ and $[M], [N]\in\M_k$,
\[
\Delta(f)([M],[N])=\frac{\nu^{-\langle \dim(M),\dim(N)\rangle}}{|\Ext^1(M,N)|}\sum_{\xi\in\Ext^1(M,N)}f([X_{\xi}]),
\]
where $X_{\xi}$ is the middle term of an exact sequence representing the extension of $M$ by $N$ given by $\xi$.
\item (Green's scalar product) For $[M], [N]\in\HH_{Q,k}$, define
\[
([M],[N])=\frac{\delta_{[M],[N]}}{a_M}.
\]
\end{enumerate}
These operations endow $\HH_{Q,\F_q}$ with a twisted bialgebra structure (see the Introduction). The multiplication on $\HH_{Q,\F_q}\otimes \HH_{Q,\F_q}$ is defined for homogeneous $x,y,z,w$ by $(x\otimes y)(z\otimes w)=\nu^{(y,z)}(xz\otimes yw)$. 
The comultiplication lies at the center of this paper and it is in duality with the comultilication by \eqref{hopfpairing}: it is thus necessary to give other formulas for both the multiplication and the comultiplication, more adapted for explicit computations. For this purpose, we introduce the following notations.
\medskip

For $M,N,R\in\Rep_Q\F_q$ three quiver representations, let us define
\[
F_{M,N}^{'R}=|\{(\alpha,\beta)\in\Hom(N,R)\times\Hom(R,M)\mid 0\rightarrow N\xrightarrow{\alpha}R\xrightarrow{\beta}M\rightarrow 0 \text{ is exact}\}|,
\]
\[
F_{M,N}^R=|\{X\subset R\mid X\simeq N\text{ and } R/X\simeq M\}|,
\]
\[
F^{M,N}_R=F_{M,N}^R\frac{a_Ma_N}{a_R}.
\]
The free action of $\Aut(N)\times \Aut(M)$ on $F^{'R}_{M,N}$ given by $(a,b)\cdot (\alpha,\beta)=(\alpha a^{-1},b\beta)$ gives the equality
\[
F^{'R}_{M,N}=a_Ma_NF^R_{M,N}.
\]
We have also Riedtmann's formula :
\[
F^{M,N}_R=\frac{|\Ext^1(M,N)_R|}{|\Hom(M,N)|}
\]
where $\Ext^1(M,N)_R$ is the subset of $\Ext^1(M,N)$ of extensions represented by an exact sequence with middle term isomorphic to $R$. We now have the following formulas for the multiplication and comultiplication: for $M, N, R\in\Rep_{Q}(k)$ three representations of $Q$ over $k$,
\[
[M]\star[N]=\nu^{\langle M,N\rangle}\sum_{[S]\in\M_k}F_{M,N}^S[S],
\]
\[
\Delta([R])=\sum_{[U], [V]\in\M_k}\nu^{\langle U,V\rangle} F_R^{U,V}[U]\otimes [V].
\]
The Green scalar product is a Hopf pairing, meaning that for any $f,g,h\in \HH_{Q,k}$, 
\begin{equation}\label{hopfpairing}
(fg,h)=(f\otimes g,\Delta(h)).
\end{equation}
In this formula, we have implicitly naturally defined the scalar product on $\HH_{Q,k}\otimes\HH_{Q,k}$ by the formula: 
\[
([M]\otimes[N], [R]\otimes [S])=([M],[R])([N], [S])
\]
for any $[M], [N], [R], [S]\in \HH_{Q,k}$.

\subsection{The dualization process and the Hall algebra}
We saw in the subsection \ref{dualisation} how to to dualize representations of quivers to obtain from a representation $M$ of $Q$ a representation $M^*$ of the dual quiver $Q^*$. This process induces a linear map between the Hall algebras
\[
\begin{matrix}
D : &\HH_{Q,k}&\rightarrow&\HH_{Q^*,k}\\
    & [M]&\mapsto &[M^*]
\end{matrix}
\]
which is a Hopf algebra graded anti-isomorphism. In particular, $D$ induces a linear isomorphism between the spaces of primitive elements of a quiver and its dual.

\subsection{A PBW basis for the Hall algebra}
The Hall algebra construction can be extended to any finitary exact category, see \cite{SchiffmannHall} and references therein. Let $\HH_{\mathcal{A}}$\footnote{It is an algebra and a coalgebra. This is a bialgebra if $\mathcal{A}$ is hereditary.} be the Hall algebra of the finitary exact category $\mathcal{A}$.
\begin{theorem}[Guo-Peng, Berenstein-Greenstein]\label{Guo}
Let $\mathcal{A}$ be a finitary exact category. Then for any order on the set $\ind \mathcal{A}$ of isomorphism classes of indecomposable objects in $\mathcal{A}$, $\HH_{\mathcal{\mathcal{A}}}$ is spanned, as a $\C$-vector space, by ordered monomials on $\ind \mathcal{A}$. Moreover, if $\mathcal{A}$ is Krull-Schmidt, then such monomials form a basis of $\HH_{\mathcal{A}}$.
\end{theorem}
\begin{proof}
See \cite[Theorem 2.4]{MR3463039}. See \cite[Theorem 3.1]{GuoPeng} for a quiver version.
\end{proof}

\subsection{Cuspidal functions and the theorem of Sevenhant and Van den Bergh}
In this section, we let  $Q$ be an arbitrary quiver (we allow multiple arrows, oriented cycles and edge loops) and $\F_q$ a finite field. Let $\HH_{Q,\F_q}$ be the Hall algebra of $Q$ over $\F_q$. 
\subsubsection{Cuspidal functions} We define here the objects of main interest in this paper.
\begin{definition}
An element $f\in\HH_{Q,\F_q}$ is called a \emph{cuspidal function} if it is primitive \emph{i.e.} if 
\[
\Delta{f}=f\otimes 1+1\otimes f.
\]
\end{definition}
We let $\HH_{Q,\F_q}^{\cusp}$ be the space of cuspidal functions. It decomposes as a direct sum
\[
\HH_{Q,\F_q}^{\cusp}=\bigoplus_{\dd\in\N^I}\HH_{Q,\F_q}^{\cusp}[\dd].
\]
A key fact in the proof of Theorem \ref{SVthm} below is the following.
\begin{lemma}\label{orthog}
Let $\dd\in\N^I$. Then $\HH_{Q,\F_q}^{\cusp}[\dd]$ is the $\dd$-graded component of the orthogonal with respect to Green's scalar product of the subspace
\[
\sum_{\dd',\dd''>0}\HH_{Q,\F_q}[\dd']\HH_{Q,\F_q}[\dd''].
\]
\end{lemma}

\begin{proof}
See \cite[3.1]{SevenhantVdB}.
\end{proof}

\subsubsection{The theorem of Sevenhant and Van den Bergh}
This theorem motivates the study of cuspidal functions, as they were used by Sevenhant and Van den Bergh in their article \cite{SevenhantVdB} to identify the whole Hall algebra of a quiver with the specialization at $\nu=\sqrt{q}$ of the positive part of the quantized enveloping algebra of a generalized Kac-Moody algebra associated to that quiver.
\medskip

Let $(f_j)_{j\in J}$ be a graded orthonormal basis of $\HH_{Q,\F_q}^{\cusp}$ with respect to Green's scalar product, so that in particular, for $j\in J$, $\dim f_j\in\N^I$ is well-defined. We define
\[
a_{i,j}=(\dim f_i,\dim f_j).
\]
for $i,j\in J$. By \cite[Proposition 3.2]{SevenhantVdB}, $a_{i,j}\leq 0$ for $i\neq j$ and if $a_{i,i}>0$, then $a_{i,i}=2$. The infinite matrix $(a_{i,j})$ is a generalized Cartan matrix in the sense of \cite{Borcherds}.

\begin{theorem}[Sevenhant-Van den Bergh, \cite{SevenhantVdB}]\label{SVthm} The Hall algebra $\HH_{Q,\F_q}$ is the bialgebra generated by the primitive elements $(f_j)_{j\in J}$ subject to the following relations:
\begin{enumerate}
\item For all $i,j\in J$, if $a_{i,j}=0$, then $f_if_j=f_jf_i$,
\item For all $i,j\in J$, if $a_{i,i}=2$, then
\[
\sum_{l=0}^{1-a_{i,j}}(-1)^l\left\{\begin{matrix}
									1-a_{i,j}\\
									l
                                                     \end{matrix}\right\}f_i^lf_jf_i^{1-a_{i,j}+l}=0
\]
where, for any integers $r$ and $s$, $\left\{ \begin{matrix}
										s\\
										r
										\end{matrix} \right\}$ is the $\nu$-binomial coefficient defined by :
\[
\left\{ \begin{matrix}
										s\\
										r
										\end{matrix} \right\}=\prod_{u=1}^{r}\frac{\nu^{u+s-r}-\nu^{-(u+s-r)}}{\nu^u-\nu^{-u}}
\]
and $\nu=\sqrt{q}$.
\end{enumerate}
\end{theorem}
\begin{remark}
From what we know about the Hall algebra, there is no natural choice for the basis $(f_j)_{j\in J}$. We hope to tackle this question in the future.
\end{remark}

\section{Kac and cuspidal polynomials of an affine quiver over a finite field}\label{4}
\subsection{Indecomposable and absolutely indecomposable representations count}
For a fundamental contribution on the count of representations of quivers over finite fields, see \cite{HuaCounting}. Let $Q$ be an arbitrary quiver. We denote by $A_{Q,\dd}(q)$, $\dd\in\N^I$ the Kac polynomial of $Q$ counting absolutely indecomposables representations of dimension $\dd$ over $\F_q$ and $I_{Q,\dd}(q)$ the polynomial counting indecomposable representations of $Q$ of dimension $\dd$ over $\F_q$. The following formula is well-known and is a consequence of Galois descent for quiver representations: for $\dd\in\N^I$ indivisible and $r\geq 1$,
\[
I_{Q,r\dd}(q)=\sum_{l\mid r}\frac{1}{l}\sum_{m\mid l}\mu(m)A_{Q,\frac{r}{l}\dd}(q^{\frac{l}{m}}),
\]
where $\mu$ is the M\"obius function. See also \cite[Theorem 4.1]{HuaCounting}. Sometimes, this formula is presented using plethystic operations (see \cite{BozecCounting} for basics on plethystic notation):
\[
\Exp_{z}\left(\sum_{\dd>0}I_{Q,\dd}(q)z^{\dd}\right)=\Exp_{t,z}\left(\sum_{\dd>0}A_{Q,\dd}(t)z^{\dd}\right).
\]
These polynomials do not depend on the orientation of the graph $Q$.

\subsection{Dimension count of cuspidal functions}
Let $Q$ be an arbitrary quiver. For complements on this section, see \cite{BozecCounting}. For $\dd\in\N^I$, we let
\[
C_{Q,\dd}(q)=\dim_{\C}\HH_{Q,\F_q}^{\cusp}[\dd].
\]
We will not use the following in the sequel.
\begin{theorem}[Bozec-Schiffmann, \cite{BozecCounting}]
The function $C_{Q,\dd}(q)$ is a polynomial in $\Q[q]$. 
\end{theorem}
In \emph{loc. cit.}, Bozec and Schiffmann defined the absolutely cuspidal polynomials of a quiver $Q$. They are characterized as follows.

\begin{proposition}[Bozec-Schiffmann, \cite{BozecCounting}]\label{characterization}
 The absolutely cuspidal polynomials of $Q$ form the unique family of polynomials $(C_{Q,\dd}^{abs}(t))_{\dd\in\N^I}$ satisfying the following conditions.
 \begin{enumerate}
  \item If $\dd\in\N^I$ is hyperbolic, $C_{Q,\dd}^{abs}(t)=C_{Q,\dd}(t)$,
  \item If $\dd\in\N^I$ is isotropic and indivisible, then
  \[
   \Exp_{z}\left(\sum_{r>0}C_{Q,\dd}(t)z^{\dd}\right)=\Exp_{t,z}\left(\sum_{r>0}C_{Q,\dd}^{abs}(t)z^{\dd}\right).
  \]
 \end{enumerate}
\end{proposition}

\begin{conj}[Bozec-Schiffmann, {\cite[Conjecture 1.3]{BozecCounting}}]\label{conjBS}For any $Q$ and $\dd\in\N^I$, $C_{Q,\dd}^{abs}(t)\in\N[t]$.
\end{conj}
We prove this conjecture for $\dd$ isotropic, see Corollary \ref{fin}.

\subsection{Kac and cuspidal polynomials of affine quivers}
For affine quivers, explicit expressions of the considered polynomials are known.
\subsubsection{Kac polynomials of affine quivers}
\begin{theorem}
Let $Q$ be an affine quiver, $\delta\in\N^I$ its indecomposable imaginary root, and $k=\F_q$ a finite field with $q$ elements.
\begin{enumerate}
\item If $\dd\in\N^I$ is a real root (\emph{i.e.} if $\langle \dd,\dd\rangle=1$), there is a unique indecomposable representation over $\F_q$, and it is absolutely indecomposable: $A_{Q,\dd}(q)=1=I_{Q,\dd}(q)$,
\item If $\dd\in\N^I$ is an imaginary root, then it is a positive multiple $r\delta$ of the indecomposable imaginary root and $A_{Q,\dd}(t)=t+n_0$, \emph{i.e.} there are $q+n_0$ absolutely indecomposable representations over $\F_q$ where $n_0$ is the number of vertices of $Q$ minus one (the number of vertices of the finite type quiver associated to $Q$). 
\end{enumerate}
\end{theorem}
\begin{proof}
See \cite[5.1]{HuaXiao}.
\end{proof}

\subsubsection{Cuspidal polynomials of affine quivers}
For acyclic affine quivers, \cite{HuaXiao} countains all the formulas we need. Let $Q=(I,\Omega)$ be an affine quiver.

\begin{proposition}[Hua-Xiao, {\cite[5.2]{HuaXiao}}]
For $r\geq 1$, we have:
\begin{equation}\label{cuspindec}
C_{Q,r\delta}(t)=I_{Q,r\delta}(t)-n_0
\end{equation}
and
\[
C_{Q,r\delta}^{abs}(t)=t.
\]
For $e_i=(0,\hdots,1,\hdots,0)\in\N^I$, we have
\[
 C_{Q,e_i}=C_{Q,e_i}^{abs}=1.
\]
For $\dd\not\in\{e_i : i\in I\}\cup\{r\delta : r\geq 1\}$,
\[
 C_{Q,\dd}=0.
\]

\end{proposition}

In fact, this formulas remains true for any affine quiver, \emph{i.e.} for the Jordan and cyclic quivers, as it is clear from the formulas of \cite[1.2]{BozecCounting}.

\section{Cuspidal functions of the Jordan and cyclic quivers}\label{5}
Let $Q$ be a quiver. The category of nilpotent representations of $Q$, $\Rep_Q^{\nil}(\F_q)$ is a sub-category of $\Rep_Q(\F_q)$ stable under extensions and taking subobjects. Therefore, denoting by $\M_{\F_q}^{Q,\nil}$ the set of isomorphism classes of nilpotent representations of $Q$, the subspace $\HH_{Q,\F_q}^{\nil}=\bigoplus_{[M]\in\M_{\F_q}^{Q,\nil}}\C[M]$ is in fact a sub-Hopf-algebra. Its primitive elements are called nilpotent cuspidal functions. In this section, we put the emphasis on nilpotent cuspidal functions of the Jordan and cyclic quivers. We will also see that arbitrary cuspidal functions can be expressed in terms of nilpotent cuspidal functions.

\subsection{Nilpotent cuspidal functions of the Jordan quiver}\label{nilpotentsJordan}
All in this section has been known for almost a century, as Steinitz introduced an algebra structure on the complex vector space generated by isomorphism classes of finite length modules over a discrete valuation ring with finite residue field. The link with Macdonald's ring of symmetric functions is explained in \cite{macdonald}. We recall here what we need for the convenience of the reader -- who is asked to look at the references for the proofs.
\medskip

Let $\Lambda$ be Macdonald's ring of symmetric functions. It is constructed as the direct limit in the category of graded rings of the system 
\[
(\Z[x_1,\hdots,x_n]^{\mathfrak{S}_n}, f(x_1\hdots,x_{n+1})\in\Z[x_1,\hdots, x_{n+1}]^{\mathfrak{S}_{n+1}}\rightarrow f(x_1,\hdots, x_n,0)\in\Z[x_1,\hdots,x_n]^{\mathfrak{S}_n}).
\]
The $\Z$-algebra $\Lambda$ is endowed with a structure of Hopf algebra with comultiplication $\Delta : \Lambda\longrightarrow \Lambda\otimes\Lambda$ by considering the direct limit of the system of applications
\begin{equation}
\Delta_n : \Z[x_1,\hdots,x_{2n}]^{\mathfrak{S}_{2n}}\longrightarrow \Z[x_1,\hdots,x_{n}]^{\mathfrak{S}_{n}}\otimes \Z[x_1,\hdots,x_{n}]^{\mathfrak{S}_{n}}
\end{equation} 
where $\Delta_n(x_i)=x_{i/2}\otimes 1$ if $i$ is even, and $\Delta_n(x_i)=1\otimes x_{\frac{i+1}{2}}$ if $i$ is odd. As usual, we let 
\[
e_r=\sum_{i_1<\hdots<i_r}x_{i_1}\hdots x_{i_r},
\]
\[
p_r=\sum_i x_i^r
\]
for $r\geq 1$. We do not recall the definition of Hall-Littlewood symmetric functions $P_{\lambda}(x;t)$ for a partition $\lambda$: see \cite[Chapter III]{macdonald}. The elements $p_r$ are the normalized primitive elements of $\Lambda$. We now state the main theorem of this section.

\begin{theorem}
We have an isomorphism of Hopf algebras :
\begin{equation}
\begin{matrix}
\varphi_q :& \HH_{Q_0}^{\nil}&\longrightarrow &\Lambda\otimes \C\\
&[I_{(1^r)}]&\longmapsto &q^{\frac{-r(r-1)}{2}}e_r.
\end{matrix}
\end{equation}
Moreover, for any partition $\lambda$:
\begin{equation}
\varphi_q([I_{\lambda}])=q^{-n(\lambda)}P_{\lambda}(x ;q^{-1})
\end{equation}
where $P_{\lambda}$ is the Hall-Littlewood symmetric function associated to the partition $\lambda$ and $n(\lambda)=\sum_{i}(i-1)\lambda_i$.
\end{theorem}
\begin{proof}
See \cite{macdonald}, 3.4 page 217.
\end{proof}

Define 
\[
\widetilde{p}_r=\varphi_q^{-1}(p_r)
\]
the cuspidal function of dimension $r$ of the Jordan quiver. For $m\geq 0$, define $\phi_m(t)=\prod_{i=1}^m(1-t^i)$ and let $l(\lambda)$ be the length of the partition $\lambda$. We have the following closed formula for the nilpotent cuspidal functions of the Jordan quiver.

\begin{theorem}\label{cuspJord}
The $r$-dimensional cuspidal function over $\F_q$ of $Q_0$ is:
\begin{equation}\label{cusphom}
\widetilde{p}_r=\sum_{|\lambda|=r}\phi_{l(\lambda)-1}(q)[I_{\lambda}].
\end{equation}
\end{theorem}
\begin{proof}
This formula already appears in \cite{SchiffmannNoncommutative}, formula $4.1$ or in \cite[3.2]{MR3463039}. See references therein.
\end{proof}

\subsection{Cuspidal functions for the Jordan quiver}
The following theorem provides a basis for cuspidal functions for the Jordan quiver $J$.
\begin{theorem}\label{cuspidauxJordan}
Let $a\in\lvert\AAAA^1_{\F_q}\rvert$ a point of degree $d$. Fix $t_a\in\AAAA^1_{\F_q}(\F_q)$ such that the image of $t_a : \Spec \F_{q^d}\rightarrow \AAAA^1_{\F_{q^d}}$ is $a$. We define $\widetilde{p}_{r,t_a}\in\HH_{J,\F_{q^d}}$ associated to $t_a$ as
\begin{equation}
\widetilde{p}_{r,t_a}=\sum_{|\lambda|=r}\phi_{l(\lambda)-1}(q^d)[t_aI_d+I_{\lambda}].
\end{equation}
This is a cuspidal function of $Q$ over $\F_{q^d}$. We have a linear map :
\[
F_d : \HH_{J,\F_{q^d}}\rightarrow \HH_{J,\F_q}
\]
induced by the forgetfull functor $F'_d : \Rep_J(\F_{q^d})\rightarrow \Rep_J(\F_q)$. The set
\[
\{F_d(\widetilde{p}_{r,t_a})\mid r\geq 1, |a|\in \lvert\AAAA^1_{\F_q}\rvert\}
\]
is a basis of cuspidal functions of $J$ over $\F_q$.

\end{theorem}

\begin{proof}
The functions $F_d(\widetilde{p}_{r,t_a})$ for $d\geq 1$, $r\geq 1$ and $a\in\lvert\AAAA^1_{\F_q}\rvert$ of degree $d$ are clearly linearly independent, since they have disjoint support. For $e\geq 1$, the number of such functions of dimension $e$ is the number of closed points of $\AAAA^1_{\F_q}$ of degree less or equal to $e$. This is also the number of irreducible monic polynomials in $\F_q[T]$ with degree less than or equal to $e$. This number is the dimension of $\HH_{J,\F_q}^{\cusp}[e]$, see \cite[Examples 1.2 and Proposition 4.1]{BozecCounting}. Therefore, it remains to show that for $d\geq 1$, $r\geq 1$ and $a\in\AAAA^1_{\F_q}$ a closed point of degree $d$,  $F_d(\widetilde{p}_{r,t_a})$ is indeed a cuspidal function in $\HH_{J,\F_q}$.
\medskip

Let $t\in\lvert\AAAA^1_{\F_q}\rvert$ be a closed point of degree $d$, $r$ an integer and $\lambda$ a partition of $r$. It suffices to show that the number of automorphisms of $F'_d(tI_r+J_{\lambda})$ and that of $tI_r+J_{\lambda}$ coincide, and that all subrepresentations of $F'_d(tI_r+J_{\lambda})$ are of the form $F'_d(N)$ for $N$ a subrepresentation of $tI_r+J_{\lambda}$.
\medskip

\emph{Step 1: The number of automorphisms.} To prove that the number of automorphisms of $F'_d(tI_r+J_{\lambda})$ and that of $tI_r+J_{\lambda}$ coincide, it suffices to show that an automorphism $\psi$ of $F'_d(tI_r+J_{\lambda})$ is $\F_{q^d}$-linear. But -- by definition of morphisms of quiver representations -- $\psi$ has to commute with $tI_r+J_{\lambda}$ hence with its semisimple part, which is $tI_r$. Therefore, $\psi$ commutes with the multiplication by $t$. Since $t$ is of degree $d$ over $\F_q$, $\F_q[t]=\F_{q^d}$ and $\psi$ is $\F_{q^d}$-linear, that is exactly what we wanted to prove.
\medskip

\emph{Step 2: The subrepresentations.} Let $N\subset F'_d(tI_r+J_{\lambda})$ be a subrepresentation, \emph{i.e.} $N$ is a $\F_q$-subspace of $\F_{q^d}^r$ stable under the linear map $tI_r+J_{\lambda}$. Since the semisimple part $tI_r$ of $tI_r+J_{\lambda}$ is a polynomial with coefficients in $\F_q$ without constant term of $tI_r+J_{\lambda}$, $N$ is stable under $tI_r$. Therefore, the multiplication by $t$ leaves $N$ invariant: $N$ is a $\F_{q^d}$-subspace of $\F_{q^d}^r$.
\end{proof}

\subsection{Nilpotent cuspidal functions of cyclic quivers}

Let $n\geq 2$ and $C_n$ be the cyclic quiver of length $n$.
\begin{theorem}\label{cuspCycl}
For any $r\geq 1$, $\dim_{\C}\HH_{C_n,\F_q}^{\nil\cusp}[r\delta]=1$.
\end{theorem}

\begin{proof}
In fact, Schiffmann proved the following result, see \cite[Proposition 3.25]{SchiffmannHall}. Let $U$ be the two sided ideal of $\HH_{C_n,\F_q}^{\nil}$ generated by the classes $[S_i]$ of simple representations of dimension $e_i$ for $0\leq i\leq n$ and let $\R=U^{\perp}$ be its orthogonal with respect to Green's scalar product. Then $\R$ is a sub-Hopf algebra of $\HH_{C_n,\F_q} ^{\nil}$ isomorphic to a graded polynomial ring with countably many variables
\begin{equation}\label{anneauR}
\R\simeq\C[x_j : j\geq 1]
\end{equation}
with $\deg(x_j)=j\delta$ and $\Delta(x_j)=x_j\otimes 1+1\otimes x_j$. This immediately implies our result. Indeed, if $f$ is a cuspidal function of dimension $r\delta$, then it is orthogonal to $[S_i]$ for $i\in I$ and to any nontrivial products. Therefore, $f\in \R$. By \eqref{anneauR}, a cuspidal function in $\R$ is a linear combination of $x_j$, $j\geq 1$. This proves Theorem \ref{cuspCycl}.
\end{proof}

\begin{lemma}\label{indcyclic}
 Let $d\geq 1$. Let $f\in \HH_{C_n,\F_q}^{\nil\cusp}[d\delta]$ be non-zero. Then for $I$ a nilpotent indecomposable representation of $C_n$ of dimension $d\delta$, $f(I)\neq 0$. Furthermore, $f(I)$ does not depend on the chosen indecomposable $I$.
 
 \end{lemma}

 \begin{proof}
By Theorem \ref{Guo}, $\{[I] : \quad $I$ \text{ nilpotent indecomposable}\}$ generates $\HH_{Q,\F_q}^{\nil}$ as an algebra. For $d\geq 1$, consider the orthogonal projection with respect to Green's scalar product:
  \[
  \pi : \HH_{C_n,\F_q}^{\nil}[d\delta]\rightarrow \HH_{C_n,\F_q}^{\nil\cusp}[d\delta].
  \]
By Lemma \ref{orthog}, $\pi$ restricts to a surjective linear map:
  \[
   \pi : \HH_{C_n,\F_q}^{\nil,\indec}[d\delta]\rightarrow \HH_{C_n,\F_q}^{\nil\cusp}[d\delta]
  \]
where $\HH_{C_n,\F_q}^{\nil\indec}$ is the subspace of $\HH_{C_n,\F_q}$ generated by the basis elements $[M]$ for $M$ indecomposable. Being surjective, this morphism is nonzero and there exists an indecomposable nilpotent representation $M$ of dimension $d\delta$ of $C_n$ such that $\pi([M])\neq 0$. This precisely means that $(f,[M])\neq 0$ since $\HH_{Q,\F_q}^{\nil\cusp}[d\delta]$ is one-dimensional, and therefore $f([M])\neq 0$. The group $\Z/n\Z$ acts by rotations on the quiver $C_n$, inducing an action of $\Z/n\Z$ by Hopf algebra automorphisms on $\HH_{C_n,\F_q}^{\nil}$ preserving $\HH_{C_n}^{\nil\cusp}[r\delta]$. Since the latter space is one-dimensional, $\Z/n\Z$ acts on it through a character $\Z/n\Z\rightarrow \C^*$, $i\mapsto\zeta^i$ for some $n$-th root of unity $\zeta$.
\medskip

We show that $\zeta =1$. For $M$ a representation of $C_n$ of dimension $d\delta$, $\Delta([M])(V_{0,dn-1},S_{-1})\neq 0$ if and only if $M\simeq V_{0,dn}$ or $M\simeq V_{0,dn-1}\oplus S_{-1}$ and $\Delta([M])(S_{-1},V_{0,dn-1})\neq 0$ if and only if $M\simeq V_{-1,dn}$ or $M\simeq V_{0,dn-1}\oplus S_{-1}$. Thus, we may write $f=[V_{0,nd}]+\zeta [V_{-1,nd}]+c[V_{0,nd-1}\oplus S_{-1}]+g$ where $\Delta(g)(S_{-1},V_{0,dn-1})=\Delta(g)(V_{0,dn-1},S_{-1})=0$ and $c\in\C$ is some complex number. Computing the comultiplications, we obtain :

\[
 0=\Delta(f)(V_{0,nd-1},S_{-1})=\nu^{\langle V_{0,nd-1},S_{-1}\rangle}\left(\frac{|\Aut(S_{-1})||\Aut(V_{0,nd-1})|}{|\Aut(V_{0,nd})|}+c\frac{|\Aut(S_{-1})||\Aut(V_{0,nd-1})|}{|\Aut(V_{0,nd-1}\oplus S_{-1})|}\right)
\]
and
\[
 0=\Delta(f)(S_{-1},V_{0,nd-1})=\nu^{\langle S_{-1},V_{0,nd-1}\rangle}\left(\zeta\frac{|\Aut(S_{-1})||\Aut(V_{0,nd-1})|}{|\Aut(V_{-1,nd})|}+c\frac{|\Aut(S_{-1})||\Aut(V_{0,nd-1})|}{|\Aut(I_{-1,nd-1}\oplus S_{-1})|}\right).
\]
From the first equation, it follows that $c\in\Q_{<0}$ and from the second equation, $\zeta\in\Q c$. So $\zeta\in\Q$ and $\zeta\in\R_{>0}$. This concludes the proof.

 \end{proof}
\begin{remark}
There is no known closed formula for nilpotent cuspidal functions of the Hall algebra of cyclic quivers.
\end{remark}
\subsection{Cuspidal functions of cyclic quivers}
Let $C_n$ be the cyclic quiver of length $n$. As the functor $G_n$ (see \ref{equivJC}) is an equivalence of categories, we can give an explicit formula for invertible cuspidal functions of cyclic quivers. For $a\in\lvert\AAAA^1_{\F_q}\rvert$ of degree $d$, let
\[
I_{a,\lambda}=G_n(F'_d(t_aI+J_{\lambda})),
\]
where $t_a\in\AAAA^1_{\F_q}(\F_q)$ represents $a$ (see Theorem \ref{cuspidauxJordan}).
\begin{proposition}
A basis of $\HH_{C_n,\F_q}^{\cusp\inv}$ is given by the functions
\[
f_{a,s}=\sum_{|\lambda|=s}\left(\prod_{j=1}^{l(\lambda)-1}(1-q^j)\right)[I_{a,\lambda}]
\]
for $a\in\lvert\AAAA^1_{\F_q}\rvert\setminus \{0\}$ and $s\geq 1$.
\end{proposition}
\begin{proof}
This is a consequence of Theorem \ref{cuspidauxJordan}.
\end{proof}

\section{Cuspidal functions of affine quivers}\label{6}

\subsection{Decomposition of cuspidal functions}
\begin{proposition}\label{decompcusp}
A cuspidal function $f\in \HH_{Q,k}$ in the Hall algebra of an affine quiver over a finite field decomposes as $f=f_{\PC}+f_{\IC}+f_{\RC}$ where $f_{\RC}$ (resp. $f_{\IC}$, resp. $f_{\PC}$) is a cuspidal function whose support consists of regular (resp. preinjective, resp. preprojective) representations.
\end{proposition}
\begin{proof}
Let $f\in\HH_{Q,k}^{\cusp}$ be a cuspidal function. Let $M=P\oplus R\oplus I$ be a representation of $Q$ over $k$ written as a direct sum of its preprojective, preinjective and regular summands, that is $P$ is the direct sum of the indecomposable preprojective direct summands of $M$, $I$ is the direct sum of the indecomposable preinjective direct summands of $M$ and $R$ is the direct sum of the indecomposable regular direct summands of $M$. Suppose moreover that at least two of the representations $P,R,I$ are nonzero. Then the proposition is equivalent to $f([M])=0$ for all such $M$. But, because of the properties of extensions and morphisms (see Proposition \ref{extensions}), in $\HH_{Q,k}$, we have $[M]=[P][R][I]$. This is a non trivial product, and therefore, $f$ is orthogonal to $[M]$ with respect to Green's scalar product. But this precisely means that $f([M])=0$ and implies the decomposition $f=f_{\PC}+f_{\IC}+f_{\RC}$. It remains to show that $f_{\PC}, f_{\RC}$ and $f_{\IC}$ are cuspidal. This comes from the fact that $f$ is cuspidal and $\Delta(f_{\PC})$ is supported on $\{(M,N) : \partial M+\partial N<0\}$ while $\Delta(f_{\RC})$ is supported on $\{(M,N) : \partial M+\partial N=0\}$ and $\Delta(f_{\IC})$ is supported on $\{(M,N) : \partial M+\partial N>0\}$.
\end{proof}

\begin{proposition}\label{supportreg}
Let $r\geq 1$ and $f\in\HH_{Q,\F_q}^{\cusp}[r\delta]$. Then $f$ is supported on regular representations.
\end{proposition}
\begin{proof}
Let $f=f_{\PC}+f_{\IC}+f_{\RC}$ be the decomposition given by Proposition \ref{decompcusp}. By Theorem \ref{Guo}, $\{[I] :\quad $I$ \text{ indecomposable}\}$ generates $\HH_{Q,\F_q}$ as an algebra. For $d\geq 1$, consider the orthogonal projection with respect to Green's scalar product:
  \[
  \pi : \HH_{C_n,\F_q}[r\delta]\rightarrow \HH_{C_n,\F_q}^{\cusp}[r\delta].
  \]
By Lemma \ref{orthog}, $\pi$ restricts to a surjective linear map:
  \[
   \pi : \HH_{C_n,\F_q}^{\indec}[r\delta]\rightarrow \HH_{C_n,\F_q}^{\cusp}[r\delta].
  \]
Let $M$ be an indecomposable representation of dimension $r\delta$. It is regular. Therefore, $(f_{\PC},[M])=(f_{\IC},[M])=0$ and $f_{\IC}=f_{\PC}=0$.
\end{proof}


\subsection{Regular Hall algebra of affine quivers}
In all this section, $Q$ is an acyclic affine quiver and $\F_q$ a finite field.

We denote by $\HH_{Q,\F_q,\RC}$ the subspace of $\HH_{Q,\F_q}$ generated by classes of regular representations $[M]$, for $M\in\Rep_{Q}^{\RC}(\F_q)$. Since $\Rep_Q^{\RC}(\F_q)$ is stable under extensions, this is in fact a subalgebra. We denote by $m$ the induced multiplication. The comultiplication $\Delta : \HH_{Q,k}\rightarrow\HH_{Q,k}\otimes\HH_{Q,k}$ induces a linear map
\[
\Delta_{\RC} : \HH_{Q,k,\RC}\rightarrow \HH_{Q,k,\RC}\otimes\HH_{Q,k,\RC}
\] 
which is the composition 
\[
\HH_{Q,k,\RC}\rightarrow \HH_{Q,k}\rightarrow \HH_{Q,k}\otimes \HH_{Q,k}\rightarrow \HH_{Q,k\RC}\otimes\HH_{Q,k,\RC}
\]
where the first arrow is the inclusion, the second is the comultiplication $\Delta$ and the last the orthogonal projection with respect to Green's scalar product. When elements of $\HH_{Q,\F_q}$ are seen as functions on isoclasses of representations of $Q$ over $\F_q$, the last arrow is the restriction of functions to isoclasses of regular representations. 
\medskip

The restriction of Green's scalar product to $\HH_{Q,\F_q,\RC}$ induces a hermitian scalar product $(-,-)$.

\begin{proposition}
The two operations $m$ and $\Delta_{\RC}$ endow $\HH_{Q,k,\RC}$ with a bialgebra structure and the multiplication is adjoint to the comultiplication for the restriction of Green's scalar product.
\end{proposition}

\begin{proof}
We already noticed that $m$ is an associative bilinear map. We have to prove that $\Delta_{\RC}$ is coassociative, compatible with the multiplication $m$ and that $m$ and $\Delta_{\RC}$ are adjoint for Green's scalar product. The coassociativity will follow from the associativity and the adjunction property.
\medskip

\emph{Adjunction property.} Let $M, N, R\in\Rep_{Q}^{\RC}(\F_q)$ be three regular representations of $Q$. Then,
\[
(\Delta([R]),[M]\otimes[N])=(\Delta_{\RC}([R]),[M]\otimes[N]).
\]
But by adjunction for $m$ and $\Delta$ in $\HH_{Q,\F_q}$,
\[
(\Delta([R]),[M]\otimes[N])=([R],[M][N]).
\]
Therefore,
\[
 (\Delta_{\RC}([R],[M]\otimes[N])=([R],[M][N]).
\]

This proves the adjunction for $m$ and $\Delta_{\RC}$.
\medskip

\emph{Compatibility of $m$ and $\Delta_{\RC}$.} Let $M,N\in\Rep_{Q}^{\RC}(\F_q)$. In $\HH_{Q,\F_q}$, we have the equation
\[
\Delta([M][N])=\Delta([M])\Delta([N]).
\]
Thanks to the properties of morphisms between preprojective, regular and preinjective representions (see Proposition \ref{extensions}), if $S\subset M$ is a subrepresentation, then its indecomposable summands can only be preprojective or regular. Moreover, $S$ is regular if and only if its defect $\partial S$ is zero, since preprojective representations have negative defect. Suppose $S$ is not regular. If $S'$ is an arbitrary subrepresentation of $N$, then for any regular representations $A$ and $B$, $([S]\otimes[M/S])([S']\otimes [N/S'])(A,B)=0$ since the support of $[S][S']$ contains only representations of defect $\partial S+\partial S'<0$. By reversing the role of $S$ and $S'$, we obtain the formula
\[
\Delta_{\RC}([M][N])=\Delta_{\RC}([M])\Delta_{\RC}([N]).
\]
This is precisely what we wanted.
\end{proof}

\subsection{Regular cuspidal functions of affine quivers}
A function $f\in\HH_{Q,k,\RC}$ is said to be \emph{cuspidal regular} if $\Delta_{\RC}(f)=f\otimes 1+1\otimes f$. From Proposition \ref{supportreg}, we have the inclusion
\[
\HH_{Q,k}^{\cusp}[r\delta]\subset \HH_{Q,k,\RC}^{\cusp}[r\delta]
\]
for any $r\geq 1$.

Recall the decomposition of $\Rep_Q^{\RC}(\F_q)$ into blocks $C_a^Q$ for $a\in\lvert\PP^1_{\F_q}\rvert$. As a consequence, we have the following proposition.
\begin{proposition}\label{resten}
There is an isomorphism of Hopf algebras :
\[
\HH_{Q,\F_q,\RC}\simeq \bigotimes_{a\in\lvert\PP^1_{\F_q}\rvert}\HH_{C_{a}^Q}
\]
between the regular Hall algebra and the restricted tensor product of Hall algebras of tubes.
\end{proposition}
\begin{theorem}\label{unique}
For any $a\in \lvert\PP^1_{\F_q}\rvert$ and any $n>0$, $\dim(\HH_{Q,\F_q,\RC}^{\cusp}\cap\HH_{C_a^Q}[n\deg(a)])=1$.
\end{theorem}
\begin{proof}
This is a consequence of Proposition \ref{resten}. Indeed, if $C_a^Q$ is a homogeneous tube, $\HH_{C_a^Q}\simeq \HH_{J,\F_q^{\deg(a)}}$ and if $C_{a}^Q$ is a non-homogeneous tube of period $p$, $\HH_{C_a^Q}\simeq \HH_{C_p,\F_q}$.
\end{proof}

\subsubsection{Normalization of cuspidal functions}
The space of cuspidal functions whose support is contained in a given tube is one dimensional. We give here a natural way to normalize them. Take $a\in \lvert\PP^1_{\F_q}\rvert$ a closed point and $n\geq 1$. If the corresponding tube is homogeneous, it contains exactly one indecomposable representation $I=I_{a,n}$ of dimension $n\deg(a)\delta$ and if $f$ is a cuspidal function whose support is contained in this tube, thanks to formula \eqref{cusphom}, $f(I)\neq 0$. We may normalize $f$ such that $f(I)=1$.
\medskip

If $a$ corresponds to a non-homogeneous tube of period $p$, then for $n\geq 1$, it contains $p$ indecomposables (up to isomorphism) $I_1,\hdots,I_p$ of dimension $n\delta$. By Proposition \ref{indcyclic}, if $f$ is a nonzero cuspidal function whose support is contained in this tube, then $f(I_1)=\hdots=f(I_p)\neq 0$. We may therefore normalize $f$ by fixing the value $f(I_1)=1$.

\begin{definition}
A dimension vector $\dd\in\N^I$ is said to be cuspidal if there exists a nonzero cuspidal function of dimension $\dd$. We say sometimes that a regular cuspidal function is in a given tube when its support is contained in this tube. 
\end{definition}

We define an homogeneous basis $\B$ of $\HH_{Q,\F_q,\RC}^{\cusp}$. From Theorem \ref{unique}, for each tube indexed by an element $a\in\lvert\PP^1_{\F_q}\rvert$, the space of cuspidal functions whose support is contained in this tube is one dimensional. We use the previous normalization and for $d\geq 1$, we denote by $f_{a,d}$ the unique corresponding normalized cuspidal function. We define now
\[
\mathcal{B}=\{f_{a,d} : a\in\lvert\PP^1_{\F_q}\rvert, d\geq 1\}.
\]
This is a homogeneous basis of regular cuspidal functions of $Q$.

\subsection{Comparison of dimensions}
For any $\dd=s\delta\in\N^I$, the previous theorem gives a basis of $\HH_{Q,\F_q,\RC}$ containing
\[
I_{Q,\dd}(q)-\sum_{i=1}^d(p_i-1)=\text{ number of tubes of degree $\leq s$}
\]
elements, where we recall that $d$ is the number of non-homogeneous tubes and $p_i$, $1\leq i\leq d$ denote their respective periods. Moreover, the table in Theorem \ref{ringelth} gives
\[
\sum_{i=1}^dp-i-d=n_0-1
\]
where $n_0$ is the number of vertices of $Q$ minus one. Therefore,
\[
\dim_{\C}\HH_{Q,\F_q,\RC}^{\cusp}[s\delta]=I_{Q,s\delta}(q)-n_0+1
\]
and, by Equation \eqref{cuspindec},
\[
\dim_{\C}\HH_{Q,\F_q}^{\RC}[s\delta]=C_{Q,s\delta}(q)+1
\]
which means that the subspace $\HH_{Q,\F_q}^{\cusp}[s\delta]\subset \HH_{Q,\F_q,\RC}[s\delta]$ is of codimension $1$.
\medskip

In Theorem \ref{noyau}, we will construct an explicit linear form on $\HH_{Q,\F_q,\RC}^{\cusp}[s\delta]$ whose kernel is $\HH_{Q,\F_q}^{\cusp}[s\delta]$ for any $s\geq 1$.


\subsection{Link between regular cuspidal functions of the Kronecker quiver and affine quivers}
Let $Q$ be an acyclic affine quiver. In Section \ref{Kronecker}, we defined an exact fully faithful functor
\[
 F : \Rep_{K_2}(k)\rightarrow \Rep_Q(k).
\]
This functor induces a linear map
\[
 \tilde{F} : \HH_{K_2,\F_q}\rightarrow \HH_{Q,\F_q}
\]
defined by $\tilde{F}[M]=[F(M)]$ for a representation $M$ of $K_2$ over $\F_q$. This map is an injective algebra morphism (since $F$ is fully faithful and since the essential image of $F$ is closed under extensions). It is usually not a coalgebra homomorphism, but it verifies the following property. For $f\in\HH_{Q,\F_q}$, denote by $f^{\perp}$ its orthogonal projection on $\im(\tilde{F})$ with respect to Green's scalar product. When viewing elements of $\HH_{Q,\F_q}$ as functions, $f^{\perp}$ is simply the restriction of $f$ to $\{[F(M)] : M\in\Rep_{K_2}(\F_q)\}$. Then, denoting by $\tilde{F}^{-1}$ the inverse of $\tilde{F} : \HH_{K_2,\F_q}\rightarrow \im(\tilde{F})$, we have the formula
\begin{equation}\label{projection}
 \Delta(F^{-1}(f^{\perp}))=\tilde{F}^{-1}\otimes \tilde{F}^{-1}(\Delta(f)^{\perp}).
\end{equation}
By construction the functor $F$ preserves indecomposables. Regularity of an indecomposable representation is a property of its dimension. It follows that $F$ restricts to a functor between abelian categories
\[
 F_{\RC} : \Rep_{K_2}^{\RC}(\F_q)\rightarrow \Rep_{Q}^{\RC}(\F_q)
\]
which now induces an algebra morphism between regular Hall algebras :
\[
 \tilde{F}_{\RC} : \HH_{K_2,\F_q,\RC}\rightarrow\HH_{Q,\F_q,\RC}.
\]
For $a\in\lvert\PP^1_{\F_q}\rvert$ and $n\geq 1$,
\[
 f^{K_2}_{a,n}=\sum_{|\lambda|=n}\prod_{j=1}^{l(\lambda)-1}(1-q^{j\deg(a)})[I_{a,\lambda}]
\]
are the regular cuspidal functions of the Kronecker quiver.

We first prove the following result giving the restriction of a regular cuspidal function of $\HH_{Q,\F_q,\RC}$.
\begin{proposition}\label{orth}
Let $r\geq 1$ and $f$ be the normalized regular cuspidal function of $Q$ in a given tube. Then $\tilde{F}^{-1}(f^{\perp})$ is the normalized cuspidal function of $K_2$ of a (homogeneous) tube.
\end{proposition}
\begin{proof}
In case $f$ is in a homogeneous tube, $F$ induces an equivalence of abelian categories with the corresponding tube of $\Rep_{K_2}^{\RC}(\F_q)$. Therefore, $f^{\perp}=f$, $\Delta(f)^{\perp}=\Delta(f)$ and the results follows from Equation \eqref{projection}.
\medskip

In case $f$ is in a non-homogeneous tube, let $C$ be the corresponding tube of $K_2$. Let $[M]$ be the isomorphism class of indecomposable representations of dimension $(s,s)$ contained in the tube $C$. By Lemma \ref{indcyclic}, $f([F(M)])\neq 0$ (the full faithfulness of $F$ implies that $F(M)$ is indecomposable in the same tube as $f$) and moreover by normalization, $f([F(M)])=1$. By formula \eqref{projection}, $\tilde{F}^{-1}(f^{\perp})$ is cuspidal, and moreover $\tilde{F}^{-1}(f^{\perp}([M])=1$. Therefore, $\tilde{F}^{-1}(f^{\perp})$ is the normalized cuspidal function of a homogeneous tube of $K_2$.
\end{proof}

\begin{cor}
 Let $a\in\lvert\PP^1_{\F_q}\rvert$. If $F$ sends the $a$-tube of $K_2$ on a homogeneous tube of $Q$, then 
 \[
\tilde{F}(f^{K_2}_{a,n})=f_{a,n}
\]
is the normalized cuspidal regular function of $Q$ of in the $t$-tube of dimension $n\deg(a)$.
 \medskip

 If $F$ sends the $a$-tube of $K_2$ on a non-homogeneous tube of $Q$, then
 \[
  \tilde{F}(f^{K_2}_{a,n})=f_{a,n}^{\perp}.
 \]

\end{cor}
\begin{proof}
This is an immediate consequence of the Proposition \ref{orth} since in particular, if $a\in\lvert\PP^1_{\F_q}\rvert$ corresponds to an homogeneous tube, for any $n\geq 1$, $f_{a,n}^{\perp}=f_{a,n}$.
\end{proof}

\subsection{Cuspidal functions of affine quivers}

\subsubsection{Cuspidal functions of an affine quiver in terms of regular cuspidal functions}

Let us introduce some notations concerning partitions. For $(\lambda_1,\hdots,\lambda_l)$ a partition, we define the following quantities :
\[
|\lambda|=\sum_{i=1}^l\lambda_i,
\]
\[
l(\lambda)=l
\]
and
\[
n(\lambda)=\sum_{i=1}^l(i-1)\lambda_i.
\]
For $d\geq 1$ and $q\neq 0$, set
\[
\xi(d,q)=\sum_{|\lambda|=d}\frac{\prod_{j=1}^{l(\lambda)-1}(1-q^{j})}{a_{\lambda}(q)},
\]
and $\dd\in\N^I$, define also
\[
\chi_{\dd} = \sum_{[M]\in\M_Q(\F_q)[\dd]}[M].
\]

The main theorem of this paper is the following.
\begin{theorem}\label{noyau}
Let $\mathcal{B}=\{f_{x,n} :  x\in\lvert\PP^1_{\F_q}\rvert, n\geq 1\}$ be the basis of normalized regular cuspidal functions of Q. Then the kernel of the linear form
\[
\begin{matrix}
L :& \HH_{Q,\F_q,\RC}^{\cusp}&\rightarrow &\C&\\
&f_{x,n}&\mapsto&\xi(n,q^{\deg(x)})&=(f_{x,n},\chi_{n\deg(x)})
\end{matrix}
\]
is the space of cuspidal functions of $Q$ of imaginary dimension, $\bigcup_{r\geq 1}\HH_{Q,\F_q}^{\cusp}[r\delta]$.
\end{theorem}
\begin{proof}
Let $r\geq 1$ be an integer and $\dd=r\delta$ a multiple of the indivisible imaginary root. We denote by $L[\dd] : \HH_{Q,\F_q,\RC}^{\cusp}[\dd]\rightarrow \C$ the restriction of $L$. We already know that $\HH_{Q,\F_q}^{\cusp}[\dd]\subset \HH_{Q,\F_q,\RC}^{\cusp}[\dd]$ is an hyperplane. To prove the theorem, it suffices to prove that $L[\dd]$ is a nonzero linear form whose kernel contains $\HH_{Q,\F_q}^{\cusp}[\dd]$.
\medskip

\emph{Step 1:} We prove that $\dd$-dimensional cuspidal functions of $Q$ are in the kernel of $L[\dd]$. We can suppose -- using the dualization process -- that the extending vertex $i_0$ is a sink. We use previously defined notations, in particular $I_{\theta}$ is an indecomposable of $Q$ of dimension $\theta=\delta-e_{i_0}$. We will show that for $x\in\lvert\PP^1_{\F_q}\rvert$, 
\begin{equation}\label{formxi}
L(f_{x,n})=\frac{1}{\nu^{\langle I_{\theta}^{\oplus r},S_{i_0}^{\oplus r}\rangle} |\GL_n(\F_{q^{\deg(x)}})|^2}\Delta(f_{x,n})(I_{\theta}^{\oplus r},S_{i_0}^{\oplus r}).
\end{equation}
This formula will imply our claim, since if $f$ is cuspidal of dimension $\dd$, then $\Delta(f)=f\otimes 1+1\otimes f$.
\medskip

Let us now prove Formula \eqref{formxi}. For a representation $M$ of $Q$, $\Delta([M])(I_{\theta}^{\oplus r},S_{i_0}^{\oplus r})=0$ if $M$ is not in the essential image of $F$ by Proposition \ref{ess}. Thus, since $F$ induces an equivalence of categories between $\Rep_{K_2}(\F_q)$ and the full subcategory of $\Rep_{Q}(\F_q)$ whose objects are extension of $I_{\theta}^{\oplus d_1}$ by $S_{i_0}^{\oplus d_2}$ for some nonnegative integers $d_1$ and $d_2$. It suffices therefore to prove \eqref{formxi} for $Q=K_2$ the Kronecker quiver, for which $I=S_1$ is the simple representation at the first vertex. For the Kronecker quiver,
\[
f_{x,n}=\sum_{|\lambda|=n}\prod_{j=1}^{l(\lambda)-1}(1-q^{j\deg(x)})[I_{x,\lambda}^{K_2}].
\]
Thus,
\[
\Delta(f_{x,n})(S_1^{\oplus r},S_2^{\oplus r})=\nu^{\langle S_1^{\oplus r},S_2^{\oplus r}\rangle}\sum_{|\lambda|=n}\left(\prod_{j=1}^{l(\lambda)-1}(1-q^{j\deg(x)})\right)\frac{|\Aut(S_2^{\oplus r})||\Aut(S_1^{\oplus r})|}{a_{\lambda}(q^{\deg(x)})}.
\]
Now, we have $|\Aut(S_1^{\oplus r})|=|\Aut(S_2^{\oplus r})|=|\GL_r(\F_{q^{\deg(x)}})|$, giving \eqref{formxi}.
\medskip

\emph{Step 2:} We prove that $L[\dd]$ is a nonzero linear form. If $r>1$ or $r=1$ and $q\neq 2$, there is at least one homogeneous tube of degree $r$. Let us choose $f=[S]$ where $S$ is a regular simple in such a tube. By definition, it is an element of the basis $\mathcal{B}$ of dimension $r\delta$ and $L[\dd](f)=\xi(1,q^r)=\frac{1}{q^r-1}\neq 0$, so $L[\dd]\neq 0$ in these cases. If $d=1$ and $q=2$, then in types $D$ or $E$, there are only non-homogeneous tubes in dimension $\delta$, since $\mid\PP^1_{\F_2}(\F_2)\mid$ has three elements. Let $f$ be a regular cuspidal function of dimension $\delta$ in a non-homogeneous tube. By \ref{projection} and because the essential image in dimension $\delta$ of the functor $F$ defined above is precisely the full subcategory of objects which are nontrivial extension of $I$ by $S_{i_0}$, we have 
\[
\Delta(f)(I,S_{i_0})=\Delta(f^{\perp})(I,S_{i_0})=\Delta(\tilde{F}^{-1}(f^{\perp}))(S_1,S_2)\neq 0
\]
since $\tilde{F}^{-1}(f^\perp)=[N]$ where $[N]$ is a regular cuspidal function of $K_2$ of dimension $(1,1)$, therefore $N$ is one of the following representation of $K_2$ over $\F_2$:
\[
\begin{tikzcd}
\F_2 \arrow[r,shift left,"1"] \arrow[r,shift right,swap,"0"] & \F_2
\end{tikzcd},
\quad
\begin{tikzcd}
\F_2 \arrow[r,shift left,"0"] \arrow[r,shift right,swap,"1"] & \F_2
\end{tikzcd}
\quad
\text{or}
\quad
\begin{tikzcd}
\F_2 \arrow[r,shift left,"1"] \arrow[r,shift right,swap,"1"] & \F_2.
\end{tikzcd}
\]
\end{proof}

\begin{cor}
The difference of two normalized regular cuspidal functions of the same dimension of two tubes of the same degree is a cuspidal function.
\end{cor}
\begin{proof}
This is clear since for $x\in\lvert\PP^1_{\F_q}\rvert$ and $s\geq 1$, $L(f_{x,s})$ depends only on $s$ and $\deg(x)$.
\end{proof}

\section{Two conjectures of Berenstein and Greenstein}\label{7}
In all this section, $Q$ is an acyclic affine quiver and $\F_q$ a finite field.
\subsection{Fortuitous cancellation theorem}
We prove in this section a result we use to relate Conjecture \ref{conj1} and Conjecture \ref{conj2} of Berenstein and Greenstein. We prove that for a cuspidal regular function $f\in \HH_{Q,\F_q,\RC}$, (\emph{i.e.} such that $\Delta_{\RC}(f)=f\otimes 1+1\otimes f$), the function $\Delta(f)-(f\otimes 1+1\otimes f)$ is supported on the subset of $\IC\times \PC=\{([M],[N]) : \text{ $M$ preinjective and $P$ preprojective}\}$.
\medskip

Let $f\in \HH_{Q,\F_q}$ a regular cuspidal function. \emph{A priori}, $\Delta(f)$ is a sum of the primitive part and terms of the form: $[I]\otimes[P], [I]\otimes [R\oplus P], [R\oplus I]\otimes [P]$ and $[R\oplus I]\otimes [R'\oplus P]$ where $R$ (resp. $P$, resp. $I$) is any regular (resp. preprojective, resp. preinjective) representation. Indeed, if a term of the form $[M]\otimes[N]$ appears with nonzero coefficient in the comultiplication of $[R]\in\HH_{Q,\F_q}$, this means that $N$ is a sub-representation of $R$ and $M$ the quotient $R/N$. But a subrepresentation of a regular representation has only regular and preprojective indecomposable direct summand and the quotient of a regular representation by a subobject has only indecomposable regular and preinjective summands. We will show the following cancellation theorem.
\begin{theorem}
Let $f\in\HH_{Q,\F_q,\RC}$ be a regular cuspidal function. Then the comultiplication $\Delta(f)\in\HH_{Q,\F_q}\otimes\HH_{Q,\F_q}$ is the sum of its primitive part and terms of the form $[I]\otimes[P]$.
\end{theorem}
\begin{proof}
Let $f\in\HH_{Q,\F_q,\RC}$ be a regular cuspidal function.
 The regular cuspidality precisely means that $\Delta_{\RC}(f)([M],[N])=0$ for any nonzero regular representations $M$ and $N$ of $Q$. We will use the coassociativity to prove the theorem. Suppose $\Delta(f)(R\oplus I,P)\neq 0$ for some representations $P,R$ and $I$ of $Q$ with $R$ regular, $I$ preinjective and $P$ preprojective. We first prove that for any $g\in \HH_{Q,\F_q,\RC}$, 
 \begin{equation}\label{form1}
 (\Delta\otimes 1)\circ \Delta(g)(R,I,P)=(1\otimes \Delta)(\Delta_{\RC}(g))(R,I,P)
 \end{equation}
 which will imply our assertion for the terms of the form $(R\oplus I)\otimes P$.
 \medskip

 Let $M$ be a regular representation. Then consider a filtration
 \[
 0\subset A\subset M
 \]
 of $M$ with successive quotients $P, R\oplus I$. Because $\Ext^1(R,I)=0$ and $\Hom(I,R)=0$, the datum of such a filtration is equivalent to the datum of a filtration:
 \[
 0\subset A\subset B\subset M
 \]
 with successive quotients $P, I, R$ (the two sorts of filtrations of $M$ with the given quotients are in one-to-one correspondence). Since $M$ and $M/B$ are by assumption regular, so is $B$. This proves \eqref{form1} for $g=[M]$ and then for any function $g$ by linearity. Therefore, 
 \[
(\Delta\otimes 1)\circ \Delta(f)(R,I,P)=(1\otimes \Delta)(f\otimes 1+1\otimes f)(R,I,P)=(f\otimes 1\otimes 1+1\otimes\Delta(f))(R,I,P)=0.
\]
 
 But a term of the form $[R]\otimes [I]\otimes [P]$ in the decomposition of  $(\Delta\otimes 1)\circ \Delta(f)$ can only come from the term $[R\oplus I]\otimes [P]$ of $\Delta(f)$, because $\Ext^1(R,I)=0$, yielding a contradiction if this one appears with nonzero coefficient in $\Delta(f)$.
 \medskip

 The case of $I\otimes R\oplus P$ is dual: the ingredients to handle this case are the formula
 \[
 (1\otimes \Delta)\circ \Delta(f)(I,P,R)=(\Delta\otimes 1)(\Delta_{\RC}(f))(I,P,R)
 \]
 and the fact that a term of the form $[I]\otimes[P]\otimes[R]$ in $(1\otimes \Delta)\circ \Delta(f)$ can only come from the term $[I]\otimes[R\oplus P]$ of $\Delta(f)$.
 
 The case of $[R\oplus I]\otimes [R'\oplus P]$ is more subtle but is a consequence of the formula :
 \[
 (\Delta\otimes1\otimes 1 )\circ (1\otimes \Delta)\circ \Delta(f)(R,I,P,R')=(1\otimes \Delta\otimes 1)(((1\otimes \Delta)\circ \Delta)_{\RC}(f))(R,I,P,R'),
 \]
where for $f\in\HH_{Q,\F_q}$, $((1\otimes \Delta)\circ\Delta)_{\RC}(f)$ is the projection on $\HH_{Q,\F_q,\RC}^3$ of $((1\otimes \Delta)\circ\Delta)(f)$. This formula is proved using the fact that filtrations of a regular module $M$
 \[
 0\subset A\subset M
 \]
 with successive quotients $P\oplus R'$ and $I\oplus R$ are in one-to-one correspondence with filtrations of $M$ of the form
 \[
 0\subset B\subset A\subset C\subset M
 \]
 with successive quotients $R', P, I, R$. As before, we also need the fact that a term of the form $[R]\otimes[I]\otimes[P]\otimes[R']$ in $(\Delta\otimes1\otimes 1)\circ (1\otimes \Delta)\circ \Delta(f)$ can only come from the term $[R\oplus I]\otimes [R'\oplus P]$ of $\Delta(f)$.
\end{proof}

\subsection{Two conjectures of Berenstein and Greenstein}
In their paper \cite{MR3463039}, Berenstein and Greenstein gave the following  conjectures proved by Deng and Ruan in \cite{MR3612468} using weighted projective lines and Hall polynomials.

For $n>1$, let $N(n)=\text{number of closed points of $\PP^1_{\F_q}$ of degree $n$}$ and $N(1)=q+1-d$. This is the number of closed points of $\PP^1_{\F_q}$ of degree $n$ not in $D$. 

\begin{conj}\label{conj1}
For $s\geq 1$, $d\geq 1$, and $x\in\PP_{\F_q}^1$ a closed point of degree $d$ not in $D$,
\[
f_{x,s}-\frac{1}{N(d)}\sum_{\substack{y\in\lvert\PP^1_{\F_q}\rvert\setminus D\\ \deg(y)=d}}f_{y,s}
\]
is a cuspidal function.
\end{conj}

\begin{conj}\label{conj2}
Let $P$ be a preprojective representation and $I$ a preinjective representation. Then, for any partition $\lambda$ and closed points $x,y\in\lvert\PP^1_{\F_q}\rvert\setminus D$,
\[
F^{P,I}_{I_{\lambda}(x)}=F^{P,I}_{I_{\lambda}(y)}.
\]
\end{conj}

Thanks to the fortuitous cancellation Theorem, Conjecture \ref{conj2} implies Conjecture \ref{conj1}. We will prove both conjectures using Theorem \ref{noyau}. We provide a direct proof of Conjecture \ref{conj1}
\medskip

\begin{theorem}\label{solconj}
Conjecture \ref{conj1} holds.
\end{theorem}
\begin{proof}
Let $x\in\lvert\PP^1_{\F_q}\rvert\setminus D$ a closed point of degree $d$ and 
\[
f=f_{x,s}-\frac{1}{N(d)}\sum_{\substack{y\in\lvert\PP^1_{\F_q}\rvert\setminus D\\ \deg(y)=d}}f_{y,s}.
\]
Then $f$ is a function of dimension $s\deg(x)$. From Theorem \ref{noyau}, $f$ is cuspidal if and only if $L(f)=0$. But
\[
L(f)=\xi(s,q^{deg(x)})-\frac{1}{N(d)}\sum_{\substack{y\in\lvert\PP^1_{\F_q}\rvert\setminus D\\ \deg(y)=d}}\xi(s,q^{\deg(y)})=0.
\]
\end{proof}

To prove Conjecture \ref{conj2}, we first define an action of an infinite permutation group on the Hall algebra $\HH_{Q,\F_q}$.

Let $\mathfrak{S}$ be the group of degree preserving permutations of $\lvert\PP^1_{\F_q}\rvert\setminus D$. The group $\mathfrak{S}$ is isomorphic to
\[
 \prod_{e\geq 1}\mathfrak{S}_{N(e)}
\]
where for a positive integer $N$, $\mathfrak{S}_N$ is the symmetric group on $N$ letters. The action
\[
 \mathfrak{S}\rightarrow \Aut(\HH_{Q,\F_q})
\]
is defined as follows. For $M,N$ two representations, $\lambda$ a partition, $x\in\lvert\PP^1_{\F_q}\rvert\setminus D$ and $\sigma\in \mathfrak{S}$,
\begin{gather*}
 \sigma\cdot [M]=[M] \quad \text{ if $[M]$ is preprojective, preinjective or in a non-homogeneous tube}\\
 \sigma\cdot [I_{\lambda}(x)]=[I_{\lambda}(\sigma(x))]\\
 \sigma\cdot ([M]\oplus [N])=(\sigma\cdot [M])\oplus (\sigma\cdot[N])
\end{gather*}
where for notational reasons, we define here $[M]\oplus[N]=[M\oplus N]$.
It is easily seen that $\sigma$ acts as a graded linear isomorphism on $\HH_{Q,\F_q}$.

The following is the second important result of this paper.
\begin{theorem}
 The group $\mathfrak{S}$ acts by Hopf-algebra automorphisms on $\HH_{Q,\F_q}$.
\end{theorem}
\begin{proof}
 It is a consequence of the following facts.
 \begin{enumerate}
 \item $\sigma$ acts as an isometry of $\HH_{Q,\F_q}$,
 \item The action of $\sigma$ leaves $\HH_{Q,\F_q,\RC}$ and $\HH_{Q,\F_q,\RC}^{\cusp}$ stable,
 \item $\sigma$ commutes with the linear form $L$. In particular, it preserves $\HH_{Q,\F_q}^{\cusp}[\dd]$ for any dimension $\dd$.
\end{enumerate}
Indeed, for $(1)$, we have to prove that for $M$ a representation of $Q$ over $\F_q$, the number of automorphisms of $M$ and $\sigma M$ are the same. We write $M\simeq I\oplus R\oplus P$ with $I$ preinjective, $R$ regular and $P$ preprojective. There is no morphisms from $I$ to $R$ or $P$ and no morphisms from $R$ to $P$. Therefore, an   endomorphism of $M$ is upper triangular in this direct sum decomposition. Since for a partition $\lambda$ and a closed point $x\in\lvert\PP^1_{\F_q}\rvert\setminus D$ the number of automorphisms of $I_{\lambda}(x)$, $a_{\lambda}(q^{\deg(x)})$, only depends on the degree of $x$, we only have to prove that for a partition $\lambda$, $x\in\lvert\PP^1_{\F_q}\rvert\setminus D$, a preprojective representation $P$ and a preinjective representation $I$, both numbers
\begin{gather*}
\dim\Hom(P,I_{\lambda}(x)),\\
\dim\Hom(I_{\lambda}(x),I)
\end{gather*}
only depend on the degree of $x$.
This is straightforward since using the Euler form, the first equals $\langle\dim P,|\lambda|\deg(x)\delta\rangle$ and the second equals $\langle|\lambda|\deg(x)\delta,\dim I\rangle$.

For $(2)$, by definition $\sigma$ let $\HH_{Q,\F_q,\RC}$ stable and the element $f_{x,n}$ of $\HH_{Q,\F_q,\RC}$ verifies $\sigma\cdot f_{x,n}=f_{\sigma(x),n}$, so $\HH_{Q,\F_q,\RC}^{\cusp}$ is stable under $\sigma$. For $(3)$, we notice that $L(f_{x,n})=\xi(n,q^{\deg(x)})=L(f_{\sigma(x),n})$.

Therefore, $\sigma$ sends a system of orthogonal cuspidal generators of $\HH_{Q,\F_q}$ to an other such system. Theorem \ref{SVthm} tells us that these systems of generators satisfy the same relations: $\sigma$ is an algebra morphism. Since $\sigma$ also preserves cuspidal functions, it is a Hopf algebra automorphism.
\end{proof}

\begin{cor}
 Conjecture \ref{conj2} holds.
\end{cor}
\begin{proof}
As usual, $\Delta_{I,P}$ denotes the truncation of the comultiplication: for $f\in\HH_{Q,\F_q}$, we keep only the terms of the form $[I']\otimes[P']$ in $\Delta(f)$. A reformulation of conjecture \ref{conj2} is
\[
 \Delta_{I,P}[I_{\lambda}(x)]=\Delta_{I,P}[I_{\lambda}(y)]
\]
for any partition $\lambda$ and closed points $x,y\in \lvert\PP^1_{\F_q}\rvert\setminus D$. For any $x,y\in\lvert\PP^{1}_{\F_q}\rvert\setminus D$, we choose $\sigma\in \mathfrak{S}$ such that $\sigma(x)=y$. We have
\[
 \Delta_{I,P}([I_{\lambda}(y)])=\Delta_{I,P}(\sigma\cdot[I_{\lambda}(x)])=(\sigma\otimes\sigma)\cdot\Delta_{I,P}([I_{\lambda}(x)])=\Delta_{I,P}([I_{\lambda}(x)])
\]
as $\sigma$ acts trivially on preinjective and preprojective representations.

\end{proof}

\section{Isotropic cuspidal dimensions of quivers}\label{8}
Let $Q=(I,\Omega)$ be an arbitrary quiver and $\F_q$ a finite field.

\subsection{The support of an isotropic cuspidal dimension of a quiver}\

Write $e_i=(0,\hdots,0,1,0,\hdots, 0)$ where all coordinates except the $i$-th are $0$. The following seems to be known and is not difficult to prove but we reproduce the proof for the sake of completeness.

\begin{proposition}\label{connected}
 Let $\dd$ be a cuspidal dimension of $Q$ over $\F_q$. Then $\dd$ has a connected support.
\end{proposition}

\begin{proof}
 Let $\dd\in\N^I$ be a cuspidal dimension and $f$ a nonzero cuspidal function of dimension $\dd$. Suppose $\supp\dd$ is not connected, therefore $P=\supp\dd=P_1\sqcup P_2$ such that there is no arrow between $P_1$ and $P_2$. For a representation $M$ of $Q$ over $\F_q$, we write $M_{P_j}$ for $j=1,2$ the representation of $Q$ which coincides with $M$ on $P_j$ and is zero at the other vertices. Since $P_1$ and $P_2$ are not connected, 
 \[\Ext^1(M_{P_1},M_{P_2})=\Ext^1(M_{P_2},M_{P_1})=0
 \]
and
\[\Hom(M_{P_1},M_{P_2})=\Hom(M_{P_2},M_{P_1})=0
 .\]
 Therefore, in $\HH_{Q,\F_q}$, 
 \[
 [M_{P_1}][M_{P_2}]=[M_{P_1}\oplus M_{P_2}]=[M_{P_2}][M_{P_1}].
 \]
 Now write
 \[
  f=\sum_{[M], \dim M=\dd}c_{[M]}[M]=\sum_{[M], \dim M=\dd}c_{[M]}[M_{P_1}\oplus M_{P_2}].
 \]

 Since from the characterization of cuspidal functions, $f$ is orthogonal to any nontrivial products,
 \[
  c_{[M]}=|\Aut(M)|(f,[M])=|\Aut(M)|(f,[M_1][M_2])=0,
 \]
 giving $f=0$, contradiction.

\end{proof}

The following Theorem \ref{isotropic} is a direct consequence of \cite[Proposition 5.7]{KacInfinite} and of the inequalities verified by cuspidal dimensions in case of quivers without loops. We reproduce the proof here for arbitrary quivers for the convenience of the reader.
\begin{theorem}\label{isotropic}
Let $\dd$ be an isotropic cuspidal dimension of Q over $\F_q$ (\emph{i.e.} $(\dd,\dd)=0$ where $(-,-)$ is the symmetrized Euler form of the quiver). Then the support $\supp \dd$ of $\dd$ is an affine quiver.
\end{theorem}

\begin{proof}
From \cite[Proposition 3.2 1.(a)]{SevenhantVdB}, which can be extended to an arbitrary quiver with few modifications\footnote{Sevenhant and Van den Bergh restrict themselves to edge loop free quivers, but their proof can be slightly modified for any quiver.}, we have
\[
(\dd,e_i)\leq 0
\]
for any $i\in I$ and moreover $\supp \dd$ is connected by Proposition \ref{connected}. We set $P=\supp\dd$ and see $P$ as a full subquiver of $Q$. The condition $(\dd,\dd)=0$ then implies 
\begin{equation}\label{nul}
(\dd, e_i)=0
\end{equation}
for $i\in \supp \dd$: in fact, if $\dd=\sum_{i\in \supp\dd}d_ie_i$ with $d_i$ a positive integer, $(d,d)=\sum_{i\in\supp\dd}d_i(\dd,e_i)=0$ and each term of this sum is negative or zero, which implies the result.
\medskip

We show that if $P$ contains edge loops, $P$ is the Jordan quiver. Let $i\in \supp\dd$ be a vertex with $g$ loops. Then 
\[
(\dd,e_i)=2(1-g)d_i - \sum_{\alpha : i\rightarrow j}d_j-\sum_{\alpha : j\rightarrow i}d_j.
\] 
This quantity must be zero. The only possibility is $g=1$ and $P$ contains no vertices adjacent to $i$. Thanks to the connectedness of $P$, we deduce that $P$ is the Jordan quiver.
\medskip

Suppose now $P$ has no edge loop. Then the Cartan matrix $A_{P}$ of the subquiver $P$ is a generalized Cartan matrix as defined in \cite[Chapter 4]{KacInfinite}, that is for any $i,j$ vertices of $P$,
\[
 a_{i,i}=2 \quad\text{ and }\quad a_{i,j}=0\implies a_{j,i}=0.
\]
The matrix $A_{P}$ is indecomposable because $P$ is connected. Moreover, the vector $\dd$ has positive nonzero integer coefficients and verifies $A_{P}\dd=0$ because of equation \eqref{nul}, since by definition of the symmetrized Euler form of a quiver, $(\dd,e_i)={}^{t}e_iA\dd$. The classification of indecomposable generalized Cartan matrices given in \cite[Theorem 4.3]{KacInfinite} implies that $A_{P}$ is of affine type. From \cite[Lemma 4.5]{KacInfinite}, since $A_{P}$ is by definition symmetric, $A_{P}$ is positive semidefinite. Thanks to \cite[Theorem 8.6 2.]{SchifflerQuiver}, $\supp\dd$ is an affine quiver: this concludes the proof.
\end{proof}
\subsection{Consequences}

\subsubsection{}
Theorem \ref{isotropic} together with Theorem \ref{noyau} gives an explicit description of cuspidal isotropic functions of any quiver. The cuspidal functions of a quiver are of three different types. We have the real cuspidal functions: the $[S_i]$ for $i$ a real vertex of $Q$ and $S_i$ the simple representation of dimension $e_i$, the isotropic cuspidal functions, which are the homogeneous cuspidal functions of dimension $\dd$ with $(\dd,\dd)=0$ and finally the hyperbolic cuspidal functions, that is cuspidal functions of dimension $\dd$ such that $(\dd,\dd)<0$.

\subsubsection{A conjecture of Bozec and Schiffmann for isotropic dimensions}
In the paper, \cite{BozecCounting}, Bozec and Schiffmann give the formula:
\begin{equation}\label{teq}
 C_{Q,\dd}^{abs}(t)=t
\end{equation}
if $\dd$ is an isotropic dimension of an affine quiver. By Theorem \ref{isotropic}, we can replace the condition $(2)$ of Proposition \ref{characterization} by the condition $(2)'$: if $\dd$ is isotropic, then $C_{Q,\dd}^{abs}(t)=t$. We can also partially obtain Conjecture \ref{conjBS}.

\begin{cor}\label{fin}
Conjecture \ref{conjBS} holds for isotropic dimensions $\dd\in\N^I$.
\end{cor}









\bibliographystyle{amsplain}
\bibliography{bibliographie}
\end{document}